\documentclass{amsart}
\usepackage{amsthm, amssymb, amsmath}
\usepackage{amsfonts}  
\usepackage[english]{babel} 
\usepackage{graphicx, tikz}
\usepackage{subfigure}
\usepackage{fullpage}
\usepackage{url} 
\usepackage{gnuplot-lua-tikz}
\newtheorem{prop}{Proposition}[section] 
\newtheorem{conj}{Conjecture}[section]
\newtheorem{theorem}{Theorem}[section]
\newtheorem{lemma}{Lemma}[section]

\newtheorem{remark}{Remark}[section]
\theoremstyle{definition}

\newcommand{\F}{\mathcal F}

\newcommand{\N}{\mathbb N}
\newcommand{\Q}{\mathbb Q}
\newcommand{\Z}{\mathbb Z}

\def\supp{\text{supp}}
\begin{document}
 
\title{Matching for generalised $\beta$-transformations}
\author{Henk Bruin, Carlo Carminati, Charlene Kalle}
\address[Henk Bruin]{Faculty of Mathematics, University of Vienna, Oskar Morgensternplatz 1, 1090 Vienna, Austria}
\email[Henk Bruin]{henk.bruin@univie.ac.at}
\address[Carlo Carminati]{Dipartimento di Matematica, Universit\`a di Pisa, Largo Bruno Pontecorvo 5, I-56127, Italy}
\email[Carlo Carminati]{carminat@dm.unipi.it}
\address[Charlene Kalle]{Mathematisch Instituut, Leiden University, Niels Bohrweg 1, 2333CA Leiden, The Netherlands}
\email[Charlene Kalle]{kallecccj@math.leidenuniv.nl}


\date{Version of \today}


\maketitle
 
\begin{abstract}
We investigate matching for the family $T_\alpha(x) = \beta x + \alpha \pmod 1$, $\alpha \in [0,1]$, for fixed  $\beta > 1$.
Matching refers to the property that there is an $n \in \N$ such that $T_\alpha^n(0) = T_\alpha^n(1)$. We show that for various Pisot numbers $\beta$, matching occurs on an open dense set of $\alpha \in [0,1]$ and we compute the Hausdorff dimension of its complement. Numerical evidence shows more cases where matching is prevalent.
\end{abstract}

\section{Introduction}
When studying interval maps from the point of view of ergodic theory, the first task is to find a suitable invariant measure $\mu$, preferably one that is absolutely continuous with respect to the Lebesgue measure. If the map is piecewise linear it is easy to obtain such a measure if 
the map admits a Markov partition (see \cite{FB81}). If no Markov partition exists, the question becomes more delicate. A relatively simple expression for the invariant measure, however, can still be found if the system satisfies the so called {\em matching condition}. Matching occurs when the orbits of the left and right limits of the discontinuity (critical) points of the map meet after a number of iterations and if at that time also their (one-sided) derivatives are equal. Often results that hold for Markov maps are also valid in case the system has matching. For instance, matching causes the invariant density to be piecewise smooth or constant. For the family 
$T_\alpha : {\mathbb S}^1 \to {\mathbb S}^1$,
\begin{equation}\label{eq:T}
T_\alpha: x \mapsto \beta x + \alpha \pmod 1, \qquad \alpha \in [0,1), \, \beta \text{ fixed},
\end{equation}
formula (1.4) in \cite{FL96} states that
the $T_\alpha$-invariant density is
$$
h(x) := \frac{d\mu(x)}{dx} = 
\sum_{T^n_\alpha(0^-) < x} \beta^{-n} - \sum_{T^n_\alpha(0^+) < x} \beta^{-n}.
$$
Matching then immediately implies that $h$ is piecewise constant, but this result holds in much greater generality, see \cite{BCMP}. A further example can be found in \cite{KS12}, where a certain construction involving generalised $\beta$-transformations is proved to give a multiple tiling in case the generalised $\beta$-transformation has a Markov partition or matching.

\vskip .2cm
Especially in the case of $\alpha$-continued fraction maps, the concept of matching has proven to be very 
useful. In \cite{NN08,KSS12,CT12,CT13} matching was used to study the entropy as a function of $\alpha$. 
In \cite{DKS09} the authors used matching to give a description of the invariant density for the 
$\alpha$-Rosen continued fractions. The article \cite{BORG13} investigated the entropy and invariant 
measure for a certain two-parameter family of 
piecewise linear interval maps and conjectured about a relation between these and matching. In several parametrised families
where matching was studied, it turned out that matching occurs prevalently, i.e., for (in some sense) typical parameters. For example, in \cite{KSS12} it is shown that the set of $\alpha$'s for which the $\alpha$-continued fraction map has matching, has full Lebesgue measure. This fact is proved in \cite{CT12} as well, where it is also shown that the complement of this set has Hausdorff dimension $1$.

\vskip .2cm
For piecewise linear transformations, prevalent matching appears to be rare: for instance it can occur in the 
family \eqref{eq:T} only if the slope $\beta$ is an algebraic integer, see \eqref{q:0orbit}. It is likely that $\beta$ must 
satisfy even stronger requirements for matching to hold:  so far matching has only been observed for values 
of the slope $\beta$ that are Pisot numbers (i.e., algebraic numbers $\beta > 1$ that have all Galois 
conjugates strictly inside the unit disk) or Salem numbers (i.e., algebraic numbers $\beta > 1$ that have 
all Galois conjugates in the closed unit disk with at least one on the boundary). 
For instance if $\beta$ is the Salem number satisfying $\beta^4-\beta^3-\beta^2-\beta+1 = 0$ and if $\alpha \in \big(\frac1{\beta^4+1}, \frac{\beta}{\beta^4+1} \big)$, then the points 0 and 1 have the following orbits under $T_\alpha$:
\begin{equation}\label{eq:salem4} 
\begin{array}{ccccccc}
0 & \to & \alpha & \to & (\beta+1)\alpha & \to & (\beta^2 + \beta + 1)\alpha,\\
1 & \to & \beta + \alpha -1 & \to & (\beta+1)\alpha + \frac1{\beta}-\frac1{\beta^2} & \to & (\beta^2 + \beta + 1)\alpha - \frac1{\beta}.
\end{array}
\end{equation}
Lemma~\ref{l:1overbeta} below now implies that there is matching after four steps. For this choice of the slope numerical experiments suggests matching
is prevalent, 
but the underlying structure is not yet clear to us.

\vskip .2cm 
The object of study of this paper is the two-parameter family of circle maps from \eqref{eq:T}. These are called {\em generalised} or {\em shifted $\beta$-transformations} and they are well studied in the literature. It follows immediately from the results of Li and Yorke (\cite{LY78}) that these maps have a unique absolutely continuous invariant measure, that is ergodic. Hofbauer (\cite{Hof81b}) proved that this measure is the unique measure of maximal entropy with entropy $\log \beta$. Topological entropy is studied in \cite{FP08}. In \cite{FL96,FL97} Flatto and Lagarias studied the lap-counting function of this map. The relation with normal numbers is investigated in \cite{FP09,Sch97,Sch11}. More recently, Li, Sahlsten and Samuel studied the cases in which the corresponding $\beta$-shift is of finite type (\cite{LSS16}). 
 
\vskip .2cm
In this paper we are interested in the size of the set of $\alpha$'s for which, given a $\beta>1$, the map $T_{\alpha}: x \mapsto \beta x + \alpha \pmod 1$ has matching. We call this set $A_{\beta}$, i.e.,
\begin{equation}\label{q:A}
A_{\beta} = \{ \alpha \in [0,1]\, : \, T_{\alpha} \text{ does not have matching} \}.
\end{equation}
The first main result pertains to quadratic irrationals:
\begin{theorem}\label{t:quadratic}
 Let $\beta$ be a quadratic algebraic number. Then $T_\alpha$ exhibits matching if and only if
 $\beta$ is Pisot. In this case $\beta^2 -k\beta \pm d = 0$ for some $d \in \N$ and
 $k > d\pm 1$, and  $dim_H(A_\beta) = \frac{\log d}{\log \beta}$.
\end{theorem}
This result is proved in Section~\ref{sec:nonunit}. The methods for $+d$ and $-d$ are quite similar with the second 
family being a bit more difficult than the first.  
Section~\ref{s:minusd} includes the values of $\beta$ satifying $\beta^2-k\beta - 1$, which are sometimes 
called the metallic means. 

In Section~\ref{sec:multinacci} we make some observations about the multinacci Pisot numbers (defined as the leading root of $\beta^k - \beta^{k-1} - \dots - \beta -1 = 0$). The second main result of this paper states that in the case $k = 3$ (the tribonacci number) the non-matching set has Hausdorff dimension strictly between 0 and 1. Results on the Hausdorff dimension or Lebesgue measure for $k \ge 4$ however remain illusive.

\section{Numerical evidence}\label{s:numerical}
We have numerical evidence of prevalence of matching in the family \eqref{eq:T} for more values of $\beta$ then we have currently proofs for.
Values of the slope $\beta$ for which matching seems to be prevalent include many Pisot numbers such as the multinacci numbers and the plastic constant (root of  $x^3=x+1$). Let us mention that matching also occurs for some Salem numbers, including the famous Lehmer's constant (i.e., the root of the polynomial equation $x^{10}+x^9-x^7-x^6-x^5-x^4-x^3+x+1  = 0$,  conjecturally the smallest Salem number), even if it is still unclear whether matching is prevalent here. The same holds for the Salem number used in \eqref{eq:salem4}.

One can also use the numerical evidence to make a guess about the box dimension of the bifurcation set $A_{\beta}$ given in (\ref{q:A}). Indeed, if $A\subset [0,1]$ is a full measure open set, then the box dimension of $[0,1]\setminus A$ equals the abscissa of convergence $s^*$ of the series
$$ \Phi(s):=\sum_{J\in \mathcal{A}} |J|^s$$
where $ \mathcal{A}$ is the family of connected components of $A$. Let us fix a value $b>1$ and set $a_k:=\#\{J\in \mathcal{A} : -k-1\leq \log_b |J| < -k \}$ then $\Phi(s) \asymp \sum_k a_k b^{-sk}$ and hence $s^*=\limsup_{k\to \infty} \frac{1}{k}\log_b a_k$. This means that we can deduce information about the box dimension from a statistic of the sizes of matching intervals.
The parameter $b$ in the above construction can be chosen freely; popular choices are $b=2$, $b=e$ and $b=10$. However, sometimes a clever choice of $b$ can be the key to make the growth rate of the sequence $\log_b a_k$ apparent from its very first elements. 
Indeed, in the case of generalised $\beta$-transformations like in (\ref{eq:T}) a natural choice of the base is $b=\beta$, since it seems that if on an interval $J$ matching occurs in $k$ steps then for the size $|J|$ we have $|J| \leq c \beta^{-k}$ (see figure).
\begin{figure}[h]
\centering
\includegraphics[scale=0.45]{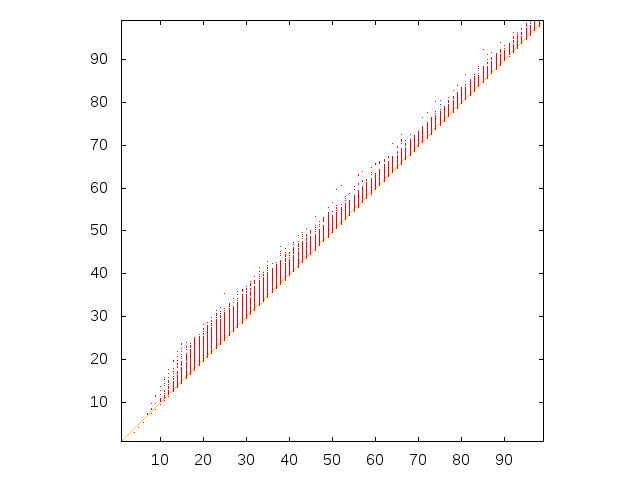}
\caption{This plot refers to the tribonacci value of $\beta$ and represents $-\log_\beta |J|$ (on the $y$-axis) against the matching index $k$ (on the $x$-axis): a linear dependence is quite apparent. 
}
\label{k-vs-size}
\end{figure}

What actually happens in these constant slope cases is that there are many matching intervals of the very same size, and these ``modal'' sizes form a sequence $s_n \sim \beta^{-n}$. 
This phenomenon is most evident when $\beta$ is a quadratic irrational; for instance in the case $\beta=2+\sqrt{2}$ all matching intervals seem to follow a very regular pattern, namely there is a decreasing sequence $s_n$ (the ``sizes'') such that:
\begin{enumerate}
\item[(i)] $-\frac{1}{n}\log_\beta s_n \to 1$  as $n\to \infty$ (note that here the log is in base $\beta=2+\sqrt{2}$);
\item[(ii)] if $J$ is a matching interval then $|J|=s_n$ for some $n$;
\item[(iii)] calling $a_n$ the cardinality of the matching intervals of size $s_n$
one gets the sequence1, 2, 6, 12, 30, 54, 126, 240, 504, 990, 2046, 4020, 8190, $\ldots$
which turns out to be known as A038199 on OEIS (see \cite{OEIS}). 
Moreover $\frac{1}{n}\log (a_n) \to 2$ as $n \to \infty$.
\end{enumerate}
   
\begin{figure}[h]
\centering
\begin{tikzpicture}[gnuplot]
\gpsolidlines
\path (0.000,0.000) rectangle (12.700,10.160);
\gpcolor{color=gp lt color border}
\gpsetlinetype{gp lt border}
\gpsetlinewidth{1.00}
\draw[gp path] (1.748,0.616)--(1.928,0.616);
\draw[gp path] (12.147,0.616)--(11.967,0.616);
\node[gp node right] at (1.564,0.616) { 1};
\draw[gp path] (1.748,1.168)--(1.838,1.168);
\draw[gp path] (12.147,1.168)--(12.057,1.168);
\draw[gp path] (1.748,1.899)--(1.838,1.899);
\draw[gp path] (12.147,1.899)--(12.057,1.899);
\draw[gp path] (1.748,2.273)--(1.838,2.273);
\draw[gp path] (12.147,2.273)--(12.057,2.273);
\draw[gp path] (1.748,2.451)--(1.928,2.451);
\draw[gp path] (12.147,2.451)--(11.967,2.451);
\node[gp node right] at (1.564,2.451) { 10};
\draw[gp path] (1.748,3.003)--(1.838,3.003);
\draw[gp path] (12.147,3.003)--(12.057,3.003);
\draw[gp path] (1.748,3.734)--(1.838,3.734);
\draw[gp path] (12.147,3.734)--(12.057,3.734);
\draw[gp path] (1.748,4.108)--(1.838,4.108);
\draw[gp path] (12.147,4.108)--(12.057,4.108);
\draw[gp path] (1.748,4.286)--(1.928,4.286);
\draw[gp path] (12.147,4.286)--(11.967,4.286);
\node[gp node right] at (1.564,4.286) { 100};
\draw[gp path] (1.748,4.838)--(1.838,4.838);
\draw[gp path] (12.147,4.838)--(12.057,4.838);
\draw[gp path] (1.748,5.569)--(1.838,5.569);
\draw[gp path] (12.147,5.569)--(12.057,5.569);
\draw[gp path] (1.748,5.943)--(1.838,5.943);
\draw[gp path] (12.147,5.943)--(12.057,5.943);
\draw[gp path] (1.748,6.121)--(1.928,6.121);
\draw[gp path] (12.147,6.121)--(11.967,6.121);
\node[gp node right] at (1.564,6.121) { 1000};
\draw[gp path] (1.748,6.673)--(1.838,6.673);
\draw[gp path] (12.147,6.673)--(12.057,6.673);
\draw[gp path] (1.748,7.404)--(1.838,7.404);
\draw[gp path] (12.147,7.404)--(12.057,7.404);
\draw[gp path] (1.748,7.778)--(1.838,7.778);
\draw[gp path] (12.147,7.778)--(12.057,7.778);
\draw[gp path] (1.748,7.956)--(1.928,7.956);
\draw[gp path] (12.147,7.956)--(11.967,7.956);
\node[gp node right] at (1.564,7.956) { 10000};
\draw[gp path] (1.748,8.508)--(1.838,8.508);
\draw[gp path] (12.147,8.508)--(12.057,8.508);
\draw[gp path] (1.748,9.239)--(1.838,9.239);
\draw[gp path] (12.147,9.239)--(12.057,9.239);
\draw[gp path] (1.748,9.613)--(1.838,9.613);
\draw[gp path] (12.147,9.613)--(12.057,9.613);
\draw[gp path] (1.748,9.791)--(1.928,9.791);
\draw[gp path] (12.147,9.791)--(11.967,9.791);
\node[gp node right] at (1.564,9.791) { 100000};
\draw[gp path] (1.748,0.616)--(1.748,0.796);
\draw[gp path] (1.748,9.791)--(1.748,9.611);
\node[gp node center] at (1.748,0.308) {-5};
\draw[gp path] (2.693,0.616)--(2.693,0.796);
\draw[gp path] (2.693,9.791)--(2.693,9.611);
\node[gp node center] at (2.693,0.308) { 0};
\draw[gp path] (3.639,0.616)--(3.639,0.796);
\draw[gp path] (3.639,9.791)--(3.639,9.611);
\node[gp node center] at (3.639,0.308) { 5};
\draw[gp path] (4.584,0.616)--(4.584,0.796);
\draw[gp path] (4.584,9.791)--(4.584,9.611);
\node[gp node center] at (4.584,0.308) { 10};
\draw[gp path] (5.529,0.616)--(5.529,0.796);
\draw[gp path] (5.529,9.791)--(5.529,9.611);
\node[gp node center] at (5.529,0.308) { 15};
\draw[gp path] (6.475,0.616)--(6.475,0.796);
\draw[gp path] (6.475,9.791)--(6.475,9.611);
\node[gp node center] at (6.475,0.308) { 20};
\draw[gp path] (7.420,0.616)--(7.420,0.796);
\draw[gp path] (7.420,9.791)--(7.420,9.611);
\node[gp node center] at (7.420,0.308) { 25};
\draw[gp path] (8.366,0.616)--(8.366,0.796);
\draw[gp path] (8.366,9.791)--(8.366,9.611);
\node[gp node center] at (8.366,0.308) { 30};
\draw[gp path] (9.311,0.616)--(9.311,0.796);
\draw[gp path] (9.311,9.791)--(9.311,9.611);
\node[gp node center] at (9.311,0.308) { 35};
\draw[gp path] (10.256,0.616)--(10.256,0.796);
\draw[gp path] (10.256,9.791)--(10.256,9.611);
\node[gp node center] at (10.256,0.308) { 40};
\draw[gp path] (11.202,0.616)--(11.202,0.796);
\draw[gp path] (11.202,9.791)--(11.202,9.611);
\node[gp node center] at (11.202,0.308) { 45};
\draw[gp path] (12.147,0.616)--(12.147,0.796);
\draw[gp path] (12.147,9.791)--(12.147,9.611);
\node[gp node center] at (12.147,0.308) { 50};
\draw[gp path] (1.748,9.791)--(1.748,0.616)--(12.147,0.616)--(12.147,9.791)--cycle;
\gpcolor{color=gp lt color 1}
\gpsetlinetype{gp lt plot 1}
\draw[gp path] (3.058,0.616)--(3.058,1.168);
\draw[gp path] (3.257,0.616)--(3.257,2.044);
\draw[gp path] (3.449,0.616)--(3.449,2.596);
\draw[gp path] (3.638,0.616)--(3.638,3.327);
\draw[gp path] (3.828,0.616)--(3.828,3.795);
\draw[gp path] (4.017,0.616)--(4.017,4.470);
\draw[gp path] (4.206,0.616)--(4.206,4.984);
\draw[gp path] (4.395,0.616)--(4.395,5.575);
\draw[gp path] (4.584,0.616)--(4.584,6.113);
\draw[gp path] (4.773,0.616)--(4.773,6.692);
\draw[gp path] (4.962,0.616)--(4.962,7.230);
\draw[gp path] (5.151,0.616)--(5.151,7.797);
\draw[gp path] (5.340,0.616)--(5.340,8.252);
\draw[gp path] (5.529,0.616)--(5.529,8.447);
\draw[gp path] (5.719,0.616)--(5.719,8.146);
\draw[gp path] (5.908,0.616)--(5.908,7.817);
\draw[gp path] (6.097,0.616)--(6.097,7.442);
\draw[gp path] (6.286,0.616)--(6.286,7.067);
\draw[gp path] (6.475,0.616)--(6.475,6.659);
\draw[gp path] (6.664,0.616)--(6.664,6.254);
\draw[gp path] (6.853,0.616)--(6.853,5.891);
\draw[gp path] (7.042,0.616)--(7.042,5.592);
\draw[gp path] (7.231,0.616)--(7.231,5.205);
\draw[gp path] (7.420,0.616)--(7.420,4.793);
\draw[gp path] (7.609,0.616)--(7.609,4.513);
\draw[gp path] (7.798,0.616)--(7.798,4.184);
\draw[gp path] (7.987,0.616)--(7.987,4.024);
\draw[gp path] (8.176,0.616)--(8.176,3.701);
\draw[gp path] (8.366,0.616)--(8.366,3.353);
\draw[gp path] (8.555,0.616)--(8.555,2.874);
\draw[gp path] (8.744,0.616)--(8.744,2.596);
\draw[gp path] (8.933,0.616)--(8.933,2.451);
\draw[gp path] (9.122,0.616)--(9.122,1.899);
\draw[gp path] (9.311,0.616)--(9.311,2.451);
\draw[gp path] (9.500,0.616)--(9.500,2.044);
\draw[gp path] (9.689,0.616)--(9.689,1.168);
\draw[gp path] (9.878,0.616)--(9.878,1.899);
\draw[gp path] (10.256,0.616)--(10.256,1.168);
\gpsetpointsize{4.00}
\gppoint{gp mark 10}{(2.587,0.616)}
\gppoint{gp mark 10}{(2.829,0.616)}
\gppoint{gp mark 10}{(3.058,1.168)}
\gppoint{gp mark 10}{(3.257,2.044)}
\gppoint{gp mark 10}{(3.449,2.596)}
\gppoint{gp mark 10}{(3.638,3.327)}
\gppoint{gp mark 10}{(3.828,3.795)}
\gppoint{gp mark 10}{(4.017,4.470)}
\gppoint{gp mark 10}{(4.206,4.984)}
\gppoint{gp mark 10}{(4.395,5.575)}
\gppoint{gp mark 10}{(4.584,6.113)}
\gppoint{gp mark 10}{(4.773,6.692)}
\gppoint{gp mark 10}{(4.962,7.230)}
\gppoint{gp mark 10}{(5.151,7.797)}
\gppoint{gp mark 10}{(5.340,8.252)}
\gppoint{gp mark 10}{(5.529,8.447)}
\gppoint{gp mark 10}{(5.719,8.146)}
\gppoint{gp mark 10}{(5.908,7.817)}
\gppoint{gp mark 10}{(6.097,7.442)}
\gppoint{gp mark 10}{(6.286,7.067)}
\gppoint{gp mark 10}{(6.475,6.659)}
\gppoint{gp mark 10}{(6.664,6.254)}
\gppoint{gp mark 10}{(6.853,5.891)}
\gppoint{gp mark 10}{(7.042,5.592)}
\gppoint{gp mark 10}{(7.231,5.205)}
\gppoint{gp mark 10}{(7.420,4.793)}
\gppoint{gp mark 10}{(7.609,4.513)}
\gppoint{gp mark 10}{(7.798,4.184)}
\gppoint{gp mark 10}{(7.987,4.024)}
\gppoint{gp mark 10}{(8.176,3.701)}
\gppoint{gp mark 10}{(8.366,3.353)}
\gppoint{gp mark 10}{(8.555,2.874)}
\gppoint{gp mark 10}{(8.744,2.596)}
\gppoint{gp mark 10}{(8.933,2.451)}
\gppoint{gp mark 10}{(9.122,1.899)}
\gppoint{gp mark 10}{(9.311,2.451)}
\gppoint{gp mark 10}{(9.500,2.044)}
\gppoint{gp mark 10}{(9.689,1.168)}
\gppoint{gp mark 10}{(9.878,1.899)}
\gppoint{gp mark 10}{(10.067,0.616)}
\gppoint{gp mark 10}{(10.256,1.168)}
\gppoint{gp mark 10}{(10.445,0.616)}
\gppoint{gp mark 10}{(10.823,0.616)}
\gppoint{gp mark 10}{(11.391,0.616)}
\gppoint{gp mark 10}{(11.580,0.616)}
\gpcolor{color=gp lt color border}
\node[gp node center] at (2.587,0.800) {\tiny  1};
\node[gp node center] at (2.829,0.800) {\tiny 1};
\node[gp node center] at (3.058,1.352) {\tiny 2};
\node[gp node center] at (3.257,2.228) {\tiny 6};
\node[gp node center] at (3.449,2.780) {\tiny 12};
\node[gp node center] at (3.638,3.511) {\tiny 30};
\node[gp node center] at (3.828,3.979) {\tiny 54};
\node[gp node center] at (4.017,4.654) {\tiny 126};
\node[gp node center] at (4.206,5.168) {\tiny 240};
\node[gp node center] at (4.395,5.759) {\tiny 504};
\node[gp node center] at (4.584,6.297) {\tiny 990};
\node[gp node center] at (4.773,6.876) {\tiny 2046};
\node[gp node center] at (4.962,7.414) {\tiny 4020};
\node[gp node center] at (5.151,7.981) {\tiny 8190};
\node[gp node center] at (5.340,8.436) {\tiny 14496};
\node[gp node center] at (5.529,8.631) {\tiny 18514};
\node[gp node center] at (5.719,8.330) {\tiny 12698};
\node[gp node center] at (5.908,8.001) {\tiny 8395};
\node[gp node center] at (6.097,7.626) {\tiny 5245};
\node[gp node center] at (6.286,7.251) {\tiny 3279};
\node[gp node center] at (6.475,6.843) {\tiny 1965};
\node[gp node center] at (6.664,6.438) {\tiny 1182};
\node[gp node center] at (6.853,6.075) {\tiny 749};
\node[gp node center] at (7.042,5.776) {\tiny 515};
\node[gp node center] at (7.231,5.389) {\tiny 317};
\node[gp node center] at (7.420,4.977) {\tiny 189};
\node[gp node center] at (7.609,4.697) {\tiny 133};
\node[gp node center] at (7.798,4.368) {\tiny 88};
\node[gp node center] at (7.987,4.208) {\tiny 72};
\node[gp node center] at (8.176,3.885) {\tiny 48};
\node[gp node center] at (8.366,3.537) {\tiny 31};
\node[gp node center] at (8.555,3.058) {\tiny 17};
\node[gp node center] at (8.744,2.780) {\tiny 12};
\node[gp node center] at (8.933,2.635) {\tiny 10};
\node[gp node center] at (9.122,2.083) {\tiny 5};
\node[gp node center] at (9.311,2.635) {\tiny 10};
\node[gp node center] at (9.500,2.228) {\tiny 6};
\node[gp node center] at (9.689,1.352) {\tiny 2};
\node[gp node center] at (9.878,2.083) {\tiny 5};
\node[gp node center] at (10.067,0.800) {\tiny 1};
\node[gp node center] at (10.256,1.352) {\tiny 2};
\node[gp node center] at (10.445,0.800) {\tiny 1};
\node[gp node center] at (10.823,0.800) {\tiny 1};
\node[gp node center] at (11.391,0.800) {\tiny 1};
\node[gp node center] at (11.580,0.800) {\tiny 1};
\gpsetlinetype{gp lt border}
\draw[gp path] (1.748,9.791)--(1.748,0.616)--(12.147,0.616)--(12.147,9.791)--cycle;
\gpdefrectangularnode{gp plot 1}{\pgfpoint{1.748cm}{0.616cm}}{\pgfpoint{12.147cm}{9.791cm}}
\end{tikzpicture}

\caption{This plot shows the frequencies of each size: each vertical bar is placed
on values $-\log_\beta s_n$, the height of each bar represents the frequency $ a_n$ and is displayed on a logarithmic scale. On
the top of each bar we have recorded the value of $a_n$: these values match perfectly with the sequence A038199 up to the 14th element (after this threshold the numerical data are likely to be incomplete). The clear decaying behaviour which becomes 
apparent after $n=15$  might also relate to the fact that, since the box dimension of the bifurcation set is smaller than one, the probability that a point taken randomly in parameter space falls in some matching interval of size  of order $\beta^{-n}$ decays exponentially as $n\to \infty$.
}
\label{frequencies}
\end{figure}
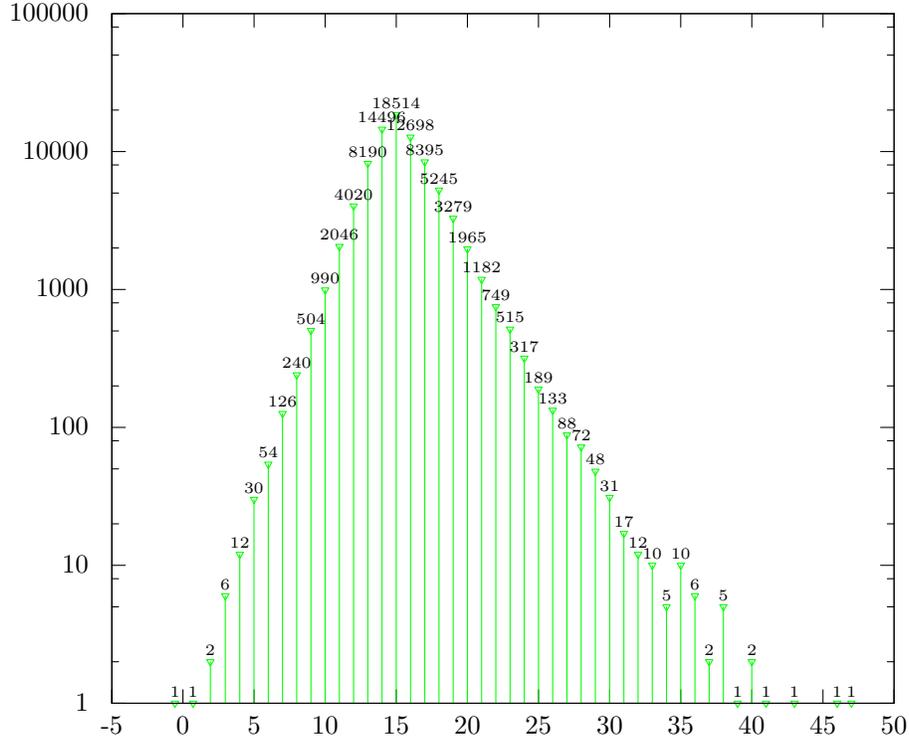

Thus one is led to conclude that the box dimension in this case is
\[ \lim_{n \to \infty}-\frac{1}{n}\log_\beta( a_n )  =  \log(2)/\log(2+\sqrt{2})=0.5644763825...
\]
This case is actually covered by  Theorem~\ref{t:quadratic}, which shows with a different argument  that for $\beta=2+\sqrt{2}$ one indeed gets that $dim_H(A_\beta) = \frac{\log 2}{\log (2+\sqrt{2})}$.

\begin{remark}\label{rm:totient}
{\rm 
For $\beta = \frac{1+\sqrt{5}}{2}$ and $s_n \approx \beta^{-n}$, we find that $a_n = \phi(n)$ is Euler's totient function. 
In Section~\ref{sec:nonunit} we relate this case to degree one circle maps $g_\alpha$ with plateaus, 
and matching with rotation numbers $\rho = \frac{m}{n} \in \Q$. This occurs for a parameter interval of length $\approx \beta^{-n}$, and the number of integers $m < n$ such that $\frac{m}{n}$ is in lowest terms is Euler's totient function. Naturally, the numerics indicate that $\dim_B(A_\beta) = 0$.
}\end{remark}

\begin{figure}[h]
\centering

\begin{tikzpicture}[gnuplot]
\gpsolidlines
\path (0.000,0.000) rectangle (12.700,5.080);
\gpcolor{color=gp lt color border}
\gpsetlinetype{gp lt border}
\gpsetlinewidth{1.00}
\draw[gp path] (1.564,0.616)--(1.744,0.616);
\draw[gp path] (12.147,0.616)--(11.967,0.616);
\node[gp node right] at (1.380,0.616) { 1};
\draw[gp path] (1.564,0.924)--(1.654,0.924);
\draw[gp path] (12.147,0.924)--(12.057,0.924);
\draw[gp path] (1.564,1.104)--(1.654,1.104);
\draw[gp path] (12.147,1.104)--(12.057,1.104);
\draw[gp path] (1.564,1.232)--(1.654,1.232);
\draw[gp path] (12.147,1.232)--(12.057,1.232);
\draw[gp path] (1.564,1.332)--(1.654,1.332);
\draw[gp path] (12.147,1.332)--(12.057,1.332);
\draw[gp path] (1.564,1.413)--(1.654,1.413);
\draw[gp path] (12.147,1.413)--(12.057,1.413);
\draw[gp path] (1.564,1.481)--(1.654,1.481);
\draw[gp path] (12.147,1.481)--(12.057,1.481);
\draw[gp path] (1.564,1.541)--(1.654,1.541);
\draw[gp path] (12.147,1.541)--(12.057,1.541);
\draw[gp path] (1.564,1.593)--(1.654,1.593);
\draw[gp path] (12.147,1.593)--(12.057,1.593);
\draw[gp path] (1.564,1.640)--(1.744,1.640);
\draw[gp path] (12.147,1.640)--(11.967,1.640);
\node[gp node right] at (1.380,1.640) { 10};
\draw[gp path] (1.564,1.948)--(1.654,1.948);
\draw[gp path] (12.147,1.948)--(12.057,1.948);
\draw[gp path] (1.564,2.128)--(1.654,2.128);
\draw[gp path] (12.147,2.128)--(12.057,2.128);
\draw[gp path] (1.564,2.256)--(1.654,2.256);
\draw[gp path] (12.147,2.256)--(12.057,2.256);
\draw[gp path] (1.564,2.355)--(1.654,2.355);
\draw[gp path] (12.147,2.355)--(12.057,2.355);
\draw[gp path] (1.564,2.436)--(1.654,2.436);
\draw[gp path] (12.147,2.436)--(12.057,2.436);
\draw[gp path] (1.564,2.505)--(1.654,2.505);
\draw[gp path] (12.147,2.505)--(12.057,2.505);
\draw[gp path] (1.564,2.564)--(1.654,2.564);
\draw[gp path] (12.147,2.564)--(12.057,2.564);
\draw[gp path] (1.564,2.617)--(1.654,2.617);
\draw[gp path] (12.147,2.617)--(12.057,2.617);
\draw[gp path] (1.564,2.664)--(1.744,2.664);
\draw[gp path] (12.147,2.664)--(11.967,2.664);
\node[gp node right] at (1.380,2.664) { 100};
\draw[gp path] (1.564,2.972)--(1.654,2.972);
\draw[gp path] (12.147,2.972)--(12.057,2.972);
\draw[gp path] (1.564,3.152)--(1.654,3.152);
\draw[gp path] (12.147,3.152)--(12.057,3.152);
\draw[gp path] (1.564,3.280)--(1.654,3.280);
\draw[gp path] (12.147,3.280)--(12.057,3.280);
\draw[gp path] (1.564,3.379)--(1.654,3.379);
\draw[gp path] (12.147,3.379)--(12.057,3.379);
\draw[gp path] (1.564,3.460)--(1.654,3.460);
\draw[gp path] (12.147,3.460)--(12.057,3.460);
\draw[gp path] (1.564,3.529)--(1.654,3.529);
\draw[gp path] (12.147,3.529)--(12.057,3.529);
\draw[gp path] (1.564,3.588)--(1.654,3.588);
\draw[gp path] (12.147,3.588)--(12.057,3.588);
\draw[gp path] (1.564,3.640)--(1.654,3.640);
\draw[gp path] (12.147,3.640)--(12.057,3.640);
\draw[gp path] (1.564,3.687)--(1.744,3.687);
\draw[gp path] (12.147,3.687)--(11.967,3.687);
\node[gp node right] at (1.380,3.687) { 1000};
\draw[gp path] (1.564,3.995)--(1.654,3.995);
\draw[gp path] (12.147,3.995)--(12.057,3.995);
\draw[gp path] (1.564,4.176)--(1.654,4.176);
\draw[gp path] (12.147,4.176)--(12.057,4.176);
\draw[gp path] (1.564,4.304)--(1.654,4.304);
\draw[gp path] (12.147,4.304)--(12.057,4.304);
\draw[gp path] (1.564,4.403)--(1.654,4.403);
\draw[gp path] (12.147,4.403)--(12.057,4.403);
\draw[gp path] (1.564,4.484)--(1.654,4.484);
\draw[gp path] (12.147,4.484)--(12.057,4.484);
\draw[gp path] (1.564,4.552)--(1.654,4.552);
\draw[gp path] (12.147,4.552)--(12.057,4.552);
\draw[gp path] (1.564,4.612)--(1.654,4.612);
\draw[gp path] (12.147,4.612)--(12.057,4.612);
\draw[gp path] (1.564,4.664)--(1.654,4.664);
\draw[gp path] (12.147,4.664)--(12.057,4.664);
\draw[gp path] (1.564,4.711)--(1.744,4.711);
\draw[gp path] (12.147,4.711)--(11.967,4.711);
\node[gp node right] at (1.380,4.711) { 10000};
\draw[gp path] (1.564,0.616)--(1.564,0.796);
\draw[gp path] (1.564,4.711)--(1.564,4.531);
\node[gp node center] at (1.564,0.308) { 0};
\draw[gp path] (3.681,0.616)--(3.681,0.796);
\draw[gp path] (3.681,4.711)--(3.681,4.531);
\node[gp node center] at (3.681,0.308) { 10};
\draw[gp path] (5.797,0.616)--(5.797,0.796);
\draw[gp path] (5.797,4.711)--(5.797,4.531);
\node[gp node center] at (5.797,0.308) { 20};
\draw[gp path] (7.914,0.616)--(7.914,0.796);
\draw[gp path] (7.914,4.711)--(7.914,4.531);
\node[gp node center] at (7.914,0.308) { 30};
\draw[gp path] (10.030,0.616)--(10.030,0.796);
\draw[gp path] (10.030,4.711)--(10.030,4.531);
\node[gp node center] at (10.030,0.308) { 40};
\draw[gp path] (12.147,0.616)--(12.147,0.796);
\draw[gp path] (12.147,4.711)--(12.147,4.531);
\node[gp node center] at (12.147,0.308) { 50};
\draw[gp path] (1.564,4.711)--(1.564,0.616)--(12.147,0.616)--(12.147,4.711)--cycle;
\node[gp node right] at (10.679,4.377) {tribonacci};
\gpcolor{color=gp lt color 6}
\gpsetlinetype{gp lt plot 6}
\draw[gp path] (10.863,4.377)--(11.779,4.377);
\draw[gp path] (1.987,0.616)--(2.199,0.616)--(2.411,0.616)--(2.834,1.332)--(3.046,1.332)%
  --(3.257,1.413)--(3.469,1.901)--(3.681,2.081)--(3.892,2.465)--(4.104,2.581)--(4.316,2.748)%
  --(4.527,2.993)--(4.739,3.206)--(4.951,3.391)--(5.162,3.536)--(5.374,3.746)--(5.586,3.552)%
  --(5.797,3.509)--(6.009,3.441)--(6.221,3.425)--(6.432,3.345)--(6.644,3.324)--(6.855,3.246)%
  --(7.067,3.361)--(7.279,3.095)--(7.490,3.080)--(7.702,3.034)--(7.914,2.917)--(8.125,2.920)%
  --(8.337,2.826)--(8.549,2.983)--(8.760,2.770)--(8.972,2.702)--(9.184,2.726)--(9.395,2.636)%
  --(9.607,2.586)--(9.819,2.536)--(10.030,2.626)--(10.242,2.421)--(10.454,2.421)--(10.665,2.373)%
  --(10.877,2.390)--(11.089,2.113)--(11.300,2.267)--(11.512,2.429)--(11.724,2.209)--(11.935,2.113)%
  --(12.147,2.047);
\gpsetpointsize{4.00}
\gppoint{gp mark 7}{(1.987,0.616)}
\gppoint{gp mark 7}{(2.199,0.616)}
\gppoint{gp mark 7}{(2.411,0.616)}
\gppoint{gp mark 7}{(2.834,1.332)}
\gppoint{gp mark 7}{(3.046,1.332)}
\gppoint{gp mark 7}{(3.257,1.413)}
\gppoint{gp mark 7}{(3.469,1.901)}
\gppoint{gp mark 7}{(3.681,2.081)}
\gppoint{gp mark 7}{(3.892,2.465)}
\gppoint{gp mark 7}{(4.104,2.581)}
\gppoint{gp mark 7}{(4.316,2.748)}
\gppoint{gp mark 7}{(4.527,2.993)}
\gppoint{gp mark 7}{(4.739,3.206)}
\gppoint{gp mark 7}{(4.951,3.391)}
\gppoint{gp mark 7}{(5.162,3.536)}
\gppoint{gp mark 7}{(5.374,3.746)}
\gppoint{gp mark 7}{(5.586,3.552)}
\gppoint{gp mark 7}{(5.797,3.509)}
\gppoint{gp mark 7}{(6.009,3.441)}
\gppoint{gp mark 7}{(6.221,3.425)}
\gppoint{gp mark 7}{(6.432,3.345)}
\gppoint{gp mark 7}{(6.644,3.324)}
\gppoint{gp mark 7}{(6.855,3.246)}
\gppoint{gp mark 7}{(7.067,3.361)}
\gppoint{gp mark 7}{(7.279,3.095)}
\gppoint{gp mark 7}{(7.490,3.080)}
\gppoint{gp mark 7}{(7.702,3.034)}
\gppoint{gp mark 7}{(7.914,2.917)}
\gppoint{gp mark 7}{(8.125,2.920)}
\gppoint{gp mark 7}{(8.337,2.826)}
\gppoint{gp mark 7}{(8.549,2.983)}
\gppoint{gp mark 7}{(8.760,2.770)}
\gppoint{gp mark 7}{(8.972,2.702)}
\gppoint{gp mark 7}{(9.184,2.726)}
\gppoint{gp mark 7}{(9.395,2.636)}
\gppoint{gp mark 7}{(9.607,2.586)}
\gppoint{gp mark 7}{(9.819,2.536)}
\gppoint{gp mark 7}{(10.030,2.626)}
\gppoint{gp mark 7}{(10.242,2.421)}
\gppoint{gp mark 7}{(10.454,2.421)}
\gppoint{gp mark 7}{(10.665,2.373)}
\gppoint{gp mark 7}{(10.877,2.390)}
\gppoint{gp mark 7}{(11.089,2.113)}
\gppoint{gp mark 7}{(11.300,2.267)}
\gppoint{gp mark 7}{(11.512,2.429)}
\gppoint{gp mark 7}{(11.724,2.209)}
\gppoint{gp mark 7}{(11.935,2.113)}
\gppoint{gp mark 7}{(12.147,2.047)}
\gppoint{gp mark 7}{(11.321,4.377)}
\gpcolor{color=gp lt color border}
\node[gp node right] at (10.679,4.069) {tetrabonacci};
\gpcolor{color=gp lt color 2}
\gpsetlinetype{gp lt plot 2}
\draw[gp path] (10.863,4.069)--(11.779,4.069);
\draw[gp path] (2.199,0.924)--(2.411,1.332)--(2.622,1.232)--(3.046,1.232)--(3.257,2.184)%
  --(3.469,2.157)--(3.681,2.197)--(3.892,2.444)--(4.104,2.706)--(4.316,2.886)--(4.527,3.002)%
  --(4.739,3.271)--(4.951,3.594)--(5.162,3.789)--(5.374,3.829)--(5.586,3.944)--(5.797,4.106)%
  --(6.009,4.269)--(6.221,4.061)--(6.432,4.022)--(6.644,4.001)--(6.855,3.966)--(7.067,3.888)%
  --(7.279,3.843)--(7.490,3.789)--(7.702,3.716)--(7.914,3.656)--(8.125,3.609)--(8.337,3.607)%
  --(8.549,3.526)--(8.760,3.429)--(8.972,3.375)--(9.184,3.349)--(9.395,3.289)--(9.607,3.238)%
  --(9.819,3.318)--(10.030,3.175)--(10.242,3.120)--(10.454,3.045)--(10.665,2.985)--(10.877,2.985)%
  --(11.089,2.899)--(11.300,2.878)--(11.512,2.816)--(11.724,2.756)--(11.935,2.752)--(12.147,2.626);
\gppoint{gp mark 11}{(2.199,0.924)}
\gppoint{gp mark 11}{(2.411,1.332)}
\gppoint{gp mark 11}{(2.622,1.232)}
\gppoint{gp mark 11}{(3.046,1.232)}
\gppoint{gp mark 11}{(3.257,2.184)}
\gppoint{gp mark 11}{(3.469,2.157)}
\gppoint{gp mark 11}{(3.681,2.197)}
\gppoint{gp mark 11}{(3.892,2.444)}
\gppoint{gp mark 11}{(4.104,2.706)}
\gppoint{gp mark 11}{(4.316,2.886)}
\gppoint{gp mark 11}{(4.527,3.002)}
\gppoint{gp mark 11}{(4.739,3.271)}
\gppoint{gp mark 11}{(4.951,3.594)}
\gppoint{gp mark 11}{(5.162,3.789)}
\gppoint{gp mark 11}{(5.374,3.829)}
\gppoint{gp mark 11}{(5.586,3.944)}
\gppoint{gp mark 11}{(5.797,4.106)}
\gppoint{gp mark 11}{(6.009,4.269)}
\gppoint{gp mark 11}{(6.221,4.061)}
\gppoint{gp mark 11}{(6.432,4.022)}
\gppoint{gp mark 11}{(6.644,4.001)}
\gppoint{gp mark 11}{(6.855,3.966)}
\gppoint{gp mark 11}{(7.067,3.888)}
\gppoint{gp mark 11}{(7.279,3.843)}
\gppoint{gp mark 11}{(7.490,3.789)}
\gppoint{gp mark 11}{(7.702,3.716)}
\gppoint{gp mark 11}{(7.914,3.656)}
\gppoint{gp mark 11}{(8.125,3.609)}
\gppoint{gp mark 11}{(8.337,3.607)}
\gppoint{gp mark 11}{(8.549,3.526)}
\gppoint{gp mark 11}{(8.760,3.429)}
\gppoint{gp mark 11}{(8.972,3.375)}
\gppoint{gp mark 11}{(9.184,3.349)}
\gppoint{gp mark 11}{(9.395,3.289)}
\gppoint{gp mark 11}{(9.607,3.238)}
\gppoint{gp mark 11}{(9.819,3.318)}
\gppoint{gp mark 11}{(10.030,3.175)}
\gppoint{gp mark 11}{(10.242,3.120)}
\gppoint{gp mark 11}{(10.454,3.045)}
\gppoint{gp mark 11}{(10.665,2.985)}
\gppoint{gp mark 11}{(10.877,2.985)}
\gppoint{gp mark 11}{(11.089,2.899)}
\gppoint{gp mark 11}{(11.300,2.878)}
\gppoint{gp mark 11}{(11.512,2.816)}
\gppoint{gp mark 11}{(11.724,2.756)}
\gppoint{gp mark 11}{(11.935,2.752)}
\gppoint{gp mark 11}{(12.147,2.626)}
\gppoint{gp mark 11}{(11.321,4.069)}
\gpcolor{color=gp lt color border}
\gpsetlinetype{gp lt border}
\draw[gp path] (1.564,4.711)--(1.564,0.616)--(12.147,0.616)--(12.147,4.711)--cycle;
\gpdefrectangularnode{gp plot 1}{\pgfpoint{1.564cm}{0.616cm}}{\pgfpoint{12.147cm}{4.711cm}}
\end{tikzpicture}
\caption{ Here we show the frequencies of sizes of matching intervals in the cases when the slope $\beta$ is the tribonacci or the tetrabonacci number: the $y$-coordinate represents (on a logarithmic scale)  the number $ a_k$  of intervals of size of order $2^{-k}$ while the $x$-coordinate represents $k$. It is quite evident that in both cases the graph shows  that $a_k$ grows exponentially in the beginning, and the decaying behaviour that can be seen after $k\geq 20$ is due to the fact that in this range our list of interval of size of order $2^{-k}$ is not complete any more.} 

\label{quabo-tribo}
\end{figure}
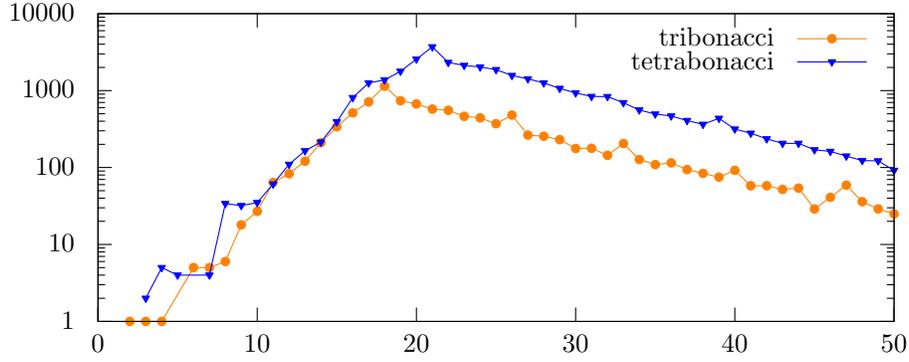

In summary, the following table gives values for the box-dimension of the bifurcation set $A_\beta$ that 
we obtained numerically. 
For the tribonacci number we prove in Section~\ref{sec:multinacci} that the Hausdorff dimension 
of $A_\beta$ is strictly between $0$ and $1$.
\[
\begin{array}{|c|r|l|}
\hline
\beta & \text{minimal polynomial} & \dim_B(A_\beta) \\[1mm]
\hline
\text{tribonacci} & \beta^3-\beta^2-\beta-1 = 0 & 0.66... \\[1mm]
\text{tetrabonacci} & \beta^4-\beta^3-\beta^2-\beta-1 = 0 & 0.76...\\[1mm]
\text{plastic} & \beta^3-\beta-1 = 0 & 0.93... \\[1mm]
\hline
\end{array}
\]

\section{The $\beta x + \alpha \pmod 1$ transformation}
For $\beta >1$ and $\alpha \in [0,1]$, the $\beta x+ \alpha$ (mod 1)-transformation is the map on ${\mathbb S}^1 = \mathbb R \slash \mathbb Z$ given by $x \mapsto \beta x + \alpha$ (mod 1). In what follows we will always assume that $\beta$ is given and we consider the family of maps $\{ T_{\alpha}: {\mathbb S}^1 \to {\mathbb S}^1 \}_{\alpha \in [0,1]}$ defined by $T_{\alpha} (x) = \beta x + \alpha$ (mod 1). The {\em critical orbits} of the map $T_{\alpha}$ are the orbits of $0^+ = \lim_{x \downarrow 0} x$ and $0^- = \lim_{x \uparrow 1}x$, i.e., the sets $\{ T_{\alpha}^n (0^+) \}_{n \ge 0}$ and $\{ T_{\alpha}^n (0^-) \}_{n \ge 0}$. For each combination of $\beta$ and $\alpha$, there is a largest integer $k\ge 0$, such that $\frac{k-\alpha}{\beta} < 1$. 
This means that for each $n \ge 1$ there are integers $a_i, b_i \in \{0, 1, \ldots, k+1\}$ such that
\begin{equation}\label{q:0orbit}
\begin{array}{rcl}
T_{\alpha}^n (0^+) &=& (\beta^{n-1} + \cdots + 1)\alpha - a_1\beta^{n-2} - \cdots - a_{n-2}\beta-a_{n-1}, \\
T_{\alpha}^n(0^-) &=& (\beta^{n-1}+\cdots + 1)\alpha + \beta^n-b_1\beta^{n-1} - \cdots - b_{n-1}\beta - b_n.
\end{array}
\end{equation}
The map $T_{\alpha}$ has a {\em Markov partition} if there exists a finite number of disjoint intervals $I_j \subseteq [0,1]$, $1\le j \le n$, such that
\begin{itemize}
\item $\overline{\bigcup_{j=1}^n I_j} = [0,1]$ and
\item for each $1 \le j,k \le n$ either $I_j \subseteq T_{\alpha}(I_k)$ or $I_j \cap T_{\alpha} (I_k) = \emptyset$.
\end{itemize}
This happens if and only if the orbits of $0^+$ and $0^-$ are finite. In particular there need to be integers $n,m$ such that $T_{\alpha}^n(0^+)=T_{\alpha}^m(0^+)$, which by the above 
implies that $\alpha \in \mathbb Q(\beta)$. If $\beta$ is an algebraic integer, then this is a countable set. It happens more frequently that $T_{\alpha}$ has matching. We say that the map $T_{\alpha}$ has {\em matching} 
if there is an $m \ge 1$ such that $T^m_{\alpha} (0^+) = T^m_{\alpha} (0^-)$. This implies the existence of an $m$ such that
\[ \beta^m - b_1\beta^{m-1} - (b_2-a_1)\beta^{m-2} - \cdots - (b_{m-1}-a_{m-2})\beta - (b_m-a_{m-1}) =0,\]
which means that $\beta$ is an algebraic integer. 

\begin{remark}{\rm
When defining matching for piecewise linear maps, one usually also asks for the derivative 
of $T^m_{\alpha}$ to be equal in the points $0^+$ and $0^-$. Since the maps $T_{\alpha}$ 
have constant slope, this condition is automatically satisfied in our case by stipulating 
that the number of iterates at matching is the same for $0^+$ and $0^-$.
}\end{remark}

Figure~\ref{f:variousalpha} illustrates that having a Markov partition does not exclude having 
matching and vice versa.

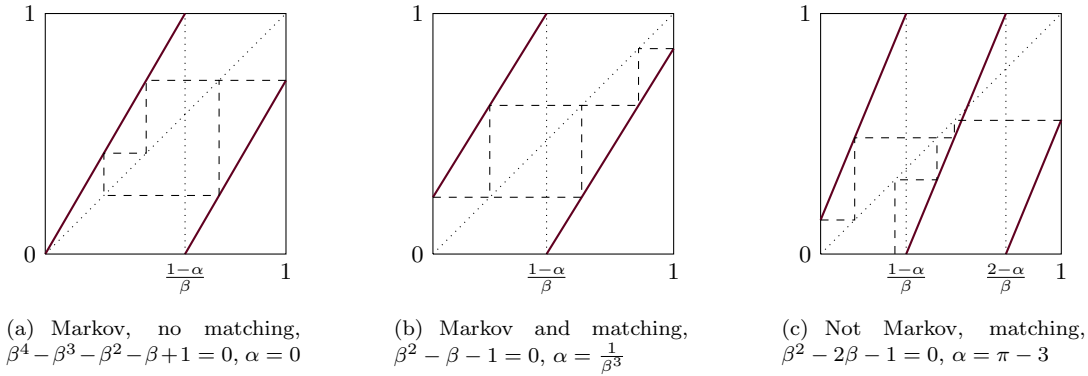
\begin{figure}[h]
\centering
\subfigure[Markov, no matching, $\beta^4-\beta^3-\beta^2-\beta+1=0$, $\alpha=0$]{
\begin{tikzpicture}[scale=3.2]
\draw(0,0)node[left]{\small $0$}--(.5807,0)node[below]{\small $\frac{1-\alpha}{\beta}$}--(1,0)node[below]{\small$1$}--(1,1)--(0,1)node[left]{\small $1$}--(0,0);
\draw[thick, purple!50!black](0,0)--(.5807,1) (.5807,0)--(1,.7221);
\draw[dotted](0,0)--(1,1)(.5807,0)--(.5807,1);
\draw[dashed](1,.7221)--(.7221,.7221)--(.7221,.2435)--(.2435,.2435)--(.2435,.4193)--(.4193,.4193)--(.4193,.7221)--(.7221,.7221);
\end{tikzpicture}}
\hspace{1cm}
\subfigure[Markov and matching, $\beta^2-\beta-1=0$, $\alpha = \frac1{\beta^3}$]{
\begin{tikzpicture}[scale=3.2]
\draw(0,0)node[left]{\small $0$}--(.4721,0)node[below]{\small $\frac{1-\alpha}{\beta}$}--(1,0)node[below]{\small$1$}--(1,1)--(0,1)node[left]{\small $1$}--(0,0);
\draw[thick, purple!50!black](0,.2361)--(.4721,1) (.4721,0)--(1,.8541);
\draw[dotted](0,0)--(1,1)(.4721,0)--(.4721,1);
\draw[dashed](0,.2361)--(.2361,.2361)--(.2361,.618)--(.618,.618)--(.618,.2361)--(.2361,.2361)(1,.8541)--(.8541,.8541)--(.8541,.618)--(.618,.618);
\end{tikzpicture}}
\hspace{1cm}
\subfigure[Not Markov, matching, $\beta^2-2\beta-1=0$, $\alpha = \pi-3$]{
\begin{tikzpicture}[scale=3.2]
\draw[white](1.1,0)--(1.1,1);
\draw(0,0)node[left]{\small $0$}--(.3556,0)node[below]{\small $\frac{1-\alpha}{\beta}$}--(.7698,0)node[below]{\small $\frac{2-\alpha}{\beta}$}--(1,0)node[below]{\small$1$}--(1,1)--(0,1)node[left]{\small $1$}--(0,0);
\draw[thick, purple!50!black](0,.1416)--(.3556,1) (.3556,0)--(.7698,1)(.7698,0)--(1,.558);
\draw[dotted](0,0)--(1,1)(.3556,0)--(.3556,1)(.7698,0)--(.7698,1);
\draw[dashed](0,.1416)--(.1416,.1416)--(.1416,.4834)--(.4834,.4834)--(.4834,.3087)--(.3087,.3087)--(.3087,0)(1,.5558)--(.5558,.5558)--(.5558,.4834)--(.4834,.4834);
\end{tikzpicture}}
\caption{The $\beta x + \alpha$ (mod 1)-transformation for various values of $\alpha$ and $\beta$. In (a) the points $0^+$ and $0^-$ are mapped to different periodic orbits by $T_{\alpha}$. In (b) we have $T^2_{\alpha}(0^-) = T^2_{\alpha}(0^+)$ and both points are part of the same 2-periodic cycle. In (c) we have taken $\alpha = \pi-3$ and $\beta$ equals the tribonacci number and we see that $T^2_{\alpha}(0^+) = T^2_{\alpha}(0^-)$.}
\label{f:variousalpha}
\end{figure}

Write
\[ \Delta(0) = \Big[0, \frac{1-\alpha}{\beta} \Big), \  \Delta(k) = \Big[ \frac{k-\alpha}{\beta}, 1 \Big], \ 
\Delta(i) = \Big[ \frac{i-\alpha}{\beta}, \frac{i+1-\alpha}{\beta} \Big), \, 1 \le i \le k-1.\]
and let $m$ denote the one-dimensional Lebesgue measure. Then $m\big( \Delta(i) \big) =\frac1{\beta}$ for all $1 \le i \le k-1$ and $m \big( \Delta(0) \big),\ m\big( \Delta(k) \big) \le \frac1{\beta}$ with the property that
\[ m \big( \Delta(0) \big) + m\big( \Delta(k) \big) = 1- \frac{k-1}{\beta}.\]
We define the cylinder sets for $T_{\alpha}$ as follows. For each $n \ge 1$ and $e_1 \cdots e_n \in \{0,1, \ldots, k\}^n$, write
\[ \Delta_{\alpha}(e_1 \cdots e_n) = \Delta(e_1 \cdots e_n) = \Delta(e_1) \cap T^{-1}_{\alpha} \Delta(e_2) \cap \cdots \cap T^{-(n-1)}_{\alpha} \Delta(e_n),\]
whenever this is non-empty. We have the following result.

\begin{lemma}\label{l:1overbeta}
Let $\beta >1$ and $\alpha \in [0,1]$ be given. Then 
$|T^n_{\alpha} (0^+) - T^n_{\alpha}(0^-) | = \frac{j}{\beta}$ for some $j \in \N$
if and only if matching occurs at iterate $n+1$.
\end{lemma}

\begin{proof}
For every $y \in [0,1)$, the preimage set $T_\alpha^{-1}(y)$ consists
of points in which every pair is $j/\beta$ apart for some $j \in \N$.
Hence it is necessary and sufficient that 
$|T^n_{\alpha} (0^+) - T^n_{\alpha}(0^-) | = j/\beta$ for matching
to occur at iterate $n+1$.
\end{proof}

\begin{remark}\label{rem:pw_affine}{\rm In what follows, to determine the value of $A_{\beta}$ we consider functions $w: [0,1] \to [0,1], \alpha \mapsto w(\alpha)$ on the parameter space of $\alpha$'s. The map $T_\alpha$ and its iterates are piecewise affine, also as function of $\alpha$. Using the chain rule repeatedly, we get
\begin{eqnarray*}
\frac{d}{d\alpha} T^n_\alpha \big( w(\alpha) \big) &=&
\frac{\partial}{\partial \alpha} T_\alpha \big(T^{n-1}_\alpha (w(\alpha))\big) +
\frac{\partial}{\partial x} T_\alpha \big(T^{n-1}_\alpha w(\alpha))\big)
\frac{d}{d\alpha} T^{n-1}_\alpha \big(w(\alpha )\big)  \\
&=& 1 + \beta \frac{d}{d\alpha} T^{n-1}_\alpha \big(w(\alpha) \big) \\
&\vdots & \nonumber \\
&=& 1+\beta + \beta^2 + \dots + \beta^{n-1} + \beta^n \frac{d}{d\alpha} w(\alpha)
= \frac{\beta^n-1}{\beta-1} + \beta^n \frac{d}{d\alpha} w(\alpha).
\end{eqnarray*}  
In particular, if $w(\alpha)$ is $n$-periodic, then  
$\frac{\beta^n-1}{\beta-1} + \beta^n \frac{d}{d\alpha} w(\alpha)
= \frac{d}{d\alpha} w(\alpha)$, so
$\frac{d}{d\alpha} w(\alpha) = - \frac{1}{\beta-1}$, independently of the period $n$.
Similarly, if $T_\alpha^n \big(w(\alpha) \big)$ is independent of $\alpha$, then 
$\frac{d}{d\alpha} w(\alpha) = - \frac{1}{\beta-1}(1-\frac{1}{\beta^n})$, whereas if 
$T_\alpha^n \big(w(\alpha)\big) \in \partial \Delta(1)$, then 
$\frac{d}{d\alpha} w(\alpha) = - \frac{1}{\beta-1}(1+\frac{1}{\beta^{n+1}})$. 
Finally, if $w(\alpha) \equiv 0^+$ is constant,
then $\frac{d}{d\alpha} T^n_\alpha ( 0^+) = \frac{\beta^n-1}{\beta-1}$.
}\end{remark}

\section{Quadratic Pisot Numbers}\label{sec:nonunit}

Solving $T^j(0^-)-T^j(0^-)=0$ using equation \eqref{q:0orbit}, we observe that
matching can only occur if $\beta$ is an algebraic integer.
In this section we look at quadratic Pisot integers; these are the leading roots
of the equations
\begin{equation}\label{eq:quap}
\beta^2-k\beta\pm d = 0, \qquad k,d \in \N,\ k > d\pm 1.
\end{equation}
The condition $k > d\pm1$ ensures that the algebraic conjugate of $\beta$ lies in
$(-1,0)$ and $(0,1)$ respectively. 
If this inequality fails, then no matching occurs:

\begin{prop}
If $\beta>1$  is an irrational quadratic integer, but not Pisot, then there is no matching.
\end{prop}

\begin{proof}
Let $\beta^2 \pm k\beta \pm d=0$, $k \geq 0$, $d \geq 1$, be the characteristic equation of $\beta$.
Since $\beta$ is a quadratic irrational, we have
$T_\alpha^j(0^-)-T_\alpha^j(0^+)= n\beta +m$ for some integers $n,m$
(which depend on $j$ and $\alpha$).
If there were matching for $T_\alpha$ then, by Lemma~\ref{l:1overbeta}, there exists $\ell \in \Z$ with
$|\ell|<\beta $ such that $n\beta+m=\ell/\beta$ which amounts to
$n\beta^2+m\beta-\ell=0$.
However, since $\beta \notin \Q$, this last equation must be an integer multiple
of the characteristic equation of $\beta$, therefore
$|\ell|=|n|d$ and thus $d<\beta$.

Note that the linear term of the characteristic equation cannot be $+k\beta$ ($k\geq0$). 
Indeed if this were the case  then the characteristic equation would lead to $\beta+k=\pm \frac{d}{\beta}$, which contradicts the fact that $\beta>1$.

Now if the characteristic equation is $\beta^2 - k\beta - d=0$, then $\beta -k=\frac{d}{\beta} \in (0,1)$. Hence
$$
 \beta = \frac{k}{2} + \sqrt{\Big(\frac{k}{2}\Big)^2 + d} < k+1
$$
which reduces to $k>d-1$ (i.e. $\beta$ is Pisot).

If on the other hand the characteristic equation is $\beta^2 - k\beta + d=0$, 
then $\beta -k=-\frac{d}{\beta} \in (-1,0)$. Hence
$$
 \beta = \frac{k}{2} + \sqrt{\Big(\frac{k}{2}\Big)^2 - d} > k-1
$$
which reduces to $k>d+1$ (i.e., $\beta$ is Pisot).

This proves that matching for some $\alpha$ forces the slope $\beta$ to be Pisot.
\end{proof}

\subsection{Case $+d$}\label{s:plusd}
In the first part of this section, let $\beta = \beta(k,d) =  \frac{k}{2} + \sqrt{(\frac{k}{2})^2-d}$ 
denote the leading root 
of $\beta^2-k\beta+d = 0$ for positive integers $k > d+1$.

\begin{theorem}
The bifurcation set $A_\beta$ has Hausdorff dimension $\dim_H(A_\beta) = \frac{\log d}{\log \beta}$.
\end{theorem}

\begin{proof}
Clearly $k-1 < \beta < k$, so $T_\alpha$ has either $k$ or $k+1$ branches depending on 
whether $\alpha < k-\beta$ or $\alpha \geq k-\beta$. More precisely,
\[ T_\alpha(0^-) - T_\alpha(0^+) = \begin{cases}
                 \beta - (k-1) & \text{ if } \alpha \in [0,k-\beta) \quad \text{($k$ branches);}\\
                 \beta - k = -\frac{d}{\beta}& \text{ if } \alpha \in [k-\beta,1) \quad \text{($k+1$ branches).}
                              \end{cases}\]
In the latter case, we have matching in the next iterate due 
to Lemma~\ref{l:1overbeta}.

\vskip .2cm 
Let $\gamma := \beta - (k-1)$ and note that $\frac{k-1-d}{\beta} < \gamma < \frac{k-d}{\beta}$, 
and $\beta\gamma - (k-1-d) = \gamma$, $\gamma-1 = -\frac{d}{\beta}$. It follows that if $x \in \Delta(i) \cap [0,1-\gamma)$, then 
\[ T_\alpha(x+\gamma) = \begin{cases}
 T_\alpha(x) + \beta \gamma -(k-d) = T_\alpha(x) - \frac{d}{\beta} 
  & \text{ if } x+\gamma \in \Delta(i+k-d)  ;\\   
  T_\alpha(x) + \beta \gamma -(k-1-d) = T_\alpha(x) + \gamma 
  & \text{ if } x+\gamma \in \Delta(i+k-1-d).  
                      \end{cases}\]
In the first case we have matching in the next step, and in 
the second case, the difference $\gamma$ remains unchanged. The transition
graph for the differences $T^j(0^-) - T^j(0^+)$ is shown in Figure~\ref{fig:plusd}

\begin{figure}[ht]
\begin{center}
\begin{tikzpicture}[scale=2]
\node[rectangle, draw=black] at (0,0) {1};
\node[rectangle, draw=black] at (1.3,0) {$\gamma$};
\node[rectangle, draw=black] at (2.7,0) {$-d/\beta$};
\node[rectangle, draw=black] at (4.6,0) {\small matching};
\draw[->] (0.3,0)--(1,0);
\draw[->] (1.5,0)--(2.3,0);
\draw[->] (3.2,0)--(4,0);
\draw[->] (1.5,0.1) .. controls (1.7,.4) and (1.4,0.5) .. (1.3,0.2);
\node at (1.67,.3) {\small $d$};
\end{tikzpicture}
\caption{The transition graph for the root of $\beta^2 - k\beta + d=0$. 
The number $d$ stands for the $d$ possible $i \in \{ 0, \dots, d-1\}$ such that 
$T_\alpha^j(0^-) \in \Delta(i)$ and $T_\alpha^j(0^+) = T_\alpha^j(0^-) + \gamma \in \Delta(i+k-1-d)$.
(In fact, $i = d$ is also possible, but $\Delta(0)$ and $\Delta(d)$ together form 
a single branch in the circle map $g_\alpha$ below.)}
\label{fig:plusd}
\end{center}
\end{figure}
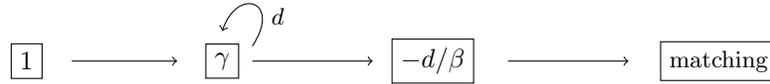

Define for $i = 0,\dots, d-1$ the ``forbidden regions''
\[ 
V_i := \{ x \in \Delta(i) : x+\gamma \in \Delta(i+k-d) \} 
= \Big[ \frac{i+k-d-\alpha}{\beta} - \gamma\ , \ \frac{i+1-\alpha}{\beta}\Big).
\]
Note that $T_\alpha(\frac{i+k-d-\alpha}{\beta} - \gamma) = k-\beta$ and 
$T_\alpha(k-\beta) = \alpha$. Define $g_\alpha:[0,k-\beta] \to [0,k-\beta]$ as
\[ g_\alpha(x) := \min\{ T_\alpha(x), k-\beta \} = \begin{cases}
               k-\beta & \text { if } x \in V := \bigcup_{i=0}^{d-1} V_i; \\
               T_\alpha(x) & \text{ otherwise.}
              \end{cases}\]


\begin{figure}[h]
\centering
\subfigure[$T_\alpha$]{
\begin{tikzpicture}[scale=3.2]
\filldraw[yellow!30] (.092,0) rectangle (.163,1);
\filldraw[yellow!30] (.325,0) rectangle (.395,1);
\filldraw[yellow!30] (.557,0) rectangle (.628,1);
\draw(0,0)node[left]{\small $0$}--(.12,0)node[below]{\small $V_0$}--(.36,0)node[below]{\small $V_1$}--(.585,0)node[below]{\small $V_2$}--(1,0)node[below]{\small$1$}--(1,1)--(0,1)node[left]{\small $1$}--(0,.697)node[left]{\small $k-\beta$}--(0,.3)node[left]{\small $\alpha$}--(0,0);
\draw[thick, purple!50!black](0,.3)--(.163,1)(.163,0)--(.395,1)(.395,0)--(.628,1)(.628,0)--(.86,1) (.86,0)--(1,.603);
\draw[dotted](0,0)--(1,1)(.163,0)--(.163,1)(.395,0)--(.395,1)(.628,0)--(.628,1)(.86,0)--(.86,1);
\draw[dashed](0,.697)--(.697,.697)--(.697,0);
\draw[dotted](.092,0)--(.092,1)(.557,0)--(.557,1)(.325,0)--(.325,1);
\end{tikzpicture}}
\hspace{1cm}
\subfigure[$g_\alpha$]{
\begin{tikzpicture}[scale=3.2]
\draw(0,0)node[left]{\small $0$}--(.12,0)node[below]{\small $V_0$}--(.36,0)node[below]{\small $V_1$}--(.585,0)node[below]{\small $V_2$}--(.697,0)--(.697,.697)--(0,.697)node[left]{\small $k-\beta$}--(0,.3)node[left]{\small $\alpha$}--(0,0);
\draw[very thick, purple!50!black](0,.3)--(.092,.697)--(.163,.697)(.163,0)--(.325,.697)--(.395,.697)(.395,0)--(.557,.697)--(.628,.697)(.628,0)--(.697,.3);
\draw[dotted](0,0)--(.697,.697)(0,.3)--(.697,.3)(.163,0)--(.163,.697)(.395,0)--(.395,.697)(.628,0)--(.628,.697)(.092,0)--(.092,.697)(.557,0)--(.557,.697)(.325,0)--(.325,.697);
\end{tikzpicture}}
\caption{The maps $T_\alpha$ and $g_\alpha$ for $\beta$ satisfying $\beta^2-5\beta+3$ and $\alpha=0.3$.}
\end{figure}
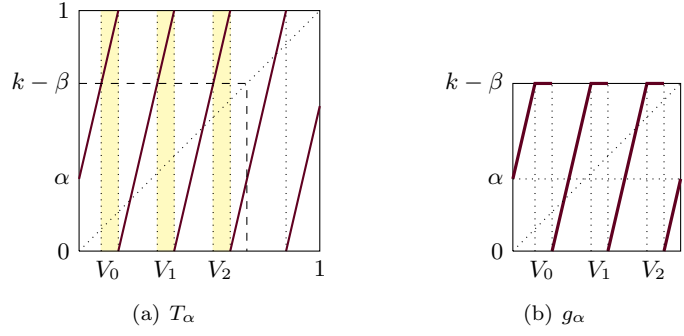

After identifying $0 \sim k-\beta$ we obtain a circle $S$
of length $m(S) = k-\beta$, and $g_\alpha:S \to S$ becomes a non-decreasing degree $d$
circle endomorphism with $d$ plateaus $V_0, \dots, V_{d-1}$, and slope $\beta$ elsewhere. Hence,
if $T_\alpha$ has no matching, then
\[ 
X_\alpha = \{ x \in S : g_\alpha^n(x) \notin \bigcup_{i=0}^{d-1} V_i \text{ for all } n \in \N\}
\]
is a $T_\alpha$-invariant set, and all invariant probability measures on it have the same 
Lyapunov exponent $\int \log T'_\alpha d\mu = \log \beta$. Using the dimension formula 
$\dim_H(\mu) = h(\mu)/\int \log T'_\alpha d\mu$ (see \cite[Proposition 4]{L81}
and \cite{Y82}), 
and maximizing the entropy over all such measures, we find $\dim_H(X_\alpha) = \frac{\log d}{\log \beta}$.

In fact, $X_\alpha$ can be covered by $a_n = O(d^n)$ intervals length $\le \beta^{-n}$.
If $T_{\alpha}$ does not have matching, then for each $n \ge 1$ there is a 
maximal interval $J=J(\alpha)$, such that $T_{\alpha}(0^+) \in J$ and 
$\frac{d}{dx}g_{\alpha}^n(x)= \beta^n$ on $J$. Hence $m(J) \le \beta^{-n}$ and moreover, 
for each point $w \in \partial J$ there is an $m \le n$ and a point $z \in \partial V$, 
such that $T_{\alpha}^m (w) = z$.

Note that the points $w$ and $z \in \partial J$ depend on $\alpha$, in the following manner.
Given a non-matching parameter, let $U$ be its neighbouhood on which the function 
$w : \alpha \mapsto w(\alpha)$ is continuous and such that there is an 
$m \in \N$ with $T_\alpha^m\big(w(\alpha)\big) =: z(\alpha) \in \partial V$
for all $\alpha \in U$. 
The definition of $V = V(\alpha)$ gives that 
$\frac{d}{d\alpha}\partial V(\alpha) = -\frac1{\beta}$. Using Remark~\ref{rem:pw_affine} we find
\[ 
\frac{d}{d\alpha} z(\alpha) = \frac{d}{d\alpha} T^m_\alpha \big(w(\alpha) \big)
=  \frac{\beta^m-1}{\beta-1} + \beta^m \frac{d}{d\alpha} w(\alpha) = -\frac1{\beta}.
\]
This implies that $\frac{d}{d\alpha} w(\alpha) =  -\frac{1}{\beta-1}(1+\frac{\beta-2}{\beta^{m+1}}) < 0$ 
for all $m \geq 0$ and $\beta > 1$. 
Hence, $J(\alpha)$ is an interval of length $\leq \beta^{-n}$ and $\partial J(\alpha)$ 
moves to the left as $\alpha$ increases. At the same time, $T_\alpha(0^+) = \alpha$ moves 
to the right with speed $1$ as $\alpha$ increases.
Therefore $U$ is an interval with $\frac1C m(J) \leq m(U) \leq C m(J)$, where $C > 0$
depends on $\beta$ but not on $\alpha$ or $U$.

This proves that the upper box dimension $\overline{\dim}_B(A_\beta) = \frac{\log d}{\log \beta}$,
and in particular $\dim_H(A_\beta) = 0$ for $d = 1$.

For the lower bound estimate of $\dim_H(A_\beta)$ and $d \geq 2$,
we introduce symbolic dynamics, assigning the labels $i = 0, \dots, d-2$ to 
the intervals
$$
Z_i := [\ (i+k-d-\alpha)/\beta\ ,\ (i+1+k-d-\alpha)/\beta) \ ),
$$
that is $V_i$ together the component of $S \setminus V$ directly
to the right of it, and the label $d-1$ to the remaining
interval: $Z_{d-1} = S \setminus \cup_{i=0}^{d-2} Z_i$.
Therefore we have $g_\alpha(Z_i) = g_\alpha(Z_i\setminus V_i) = S$.
Let $\Sigma = \{ 0, \dots, d-1\}^{\N_0}$ with metric
$d_\beta(x,y) = \beta^{1-m}$ for $m = \inf\{ i \geq 0 : x_i \neq y_i\}$.
Then each $n$-cylinder has diameter $\beta^{-n}$,
the Hausdorff dimension $\dim_H(\Sigma) = \frac{\log d}{\log \beta}$,
and the usual coding map $\pi_\alpha : S \to \Sigma$, when restricted to $X_\alpha$, 
is injective.

\begin{lemma}
 Given an $n$-cylinder $[e_0,\dots, e_{n-1}] \subset \Sigma$, there is a set
 $C_{e_0\dots e_{n-1}} \subset S$ consisting of at most $n$ half-open intervals of
 combined length $\beta^{-n} m(S)$ such that
 $\pi_\alpha(C_{e_0\dots e_{n-1}}) = [e_0,\dots, e_{n-1}]$ and
 $g_\alpha^n:C_{e_0\dots e_{n-1}} \to S$ is onto with slope $\beta^n$.
\end{lemma}

\begin{proof}
The proof is by induction. For $e_0 \in \{0, \dots, d-1\}$, let $C_{e_0} := Z_{e_0} \setminus V_{e_0}$ be 
the domain of $S\setminus V$ with label $e_0$.
This interval has length $m(S)/\beta$, is half-open and clearly $g_\alpha:C_{e_0} \to S$
is onto with slope $\beta$.

Assume now by induction that for the $n$-cylinder $[e_0,\dots, e_{n-1}]$, the set $C_{e_0\dots e_{n-1}}$ is 
constructed, with $\leq n$ half-open components $C^j_{e_0\dots e_{n-1}}$ so that $\sum_j m(C^j_{e_0\dots e_{n-1}})
= m(S) \beta^{-n}$. In particular, $g_\alpha^n(C^j_{e_0\dots e_{n-1}})$ are pairwise disjoint, since any
overlap would result in $\sum_j m(C^j_{e_0\dots e_{n-1}}) < m(S) \beta^{-n}$.

Let $e_n \in \{0, \dots, d-1\}$ be arbitrary and let
$C^j_{e_0\dots e_{n-1}e_n}$ be the subset of $C^j_{e_0\dots e_{n-1}}$ that
$g_\alpha^n$ maps into $Z_{e_n} \setminus V_{e_n}$.
Then $C^j_{e_0\dots e_{n-1}e_n}$ consists of two or one half-open intervals, depending on whether
$g_\alpha^n(C^j_{e_0\dots e_{n-1}e_n}) \owns g_\alpha(V) = n-\beta \sim 0$ or not.
But since $g_\alpha^n(C^j_{e_0\dots e_{n-1}})$ are pairwise disjoint, only
one of the  $C^j_{e_0\dots e_{n-1}e_n}$ can have two components,
and therefore $C_{e_0\dots e_{n-1}e_n} := \cup_j C^j_{e_0\dots e_{n-1}e_n}$
has at most $n+1$ components.
Furthermore $g_\alpha^{n+1}(C_{e_0\dots e_{n-1}e_n}) = S$ and the slope
is $\beta^{n+1}$. 
(See also \cite[Lemma 4.1]{BS16} for this result in a simpler context.)
\end{proof}

It follows that $\pi_\alpha$ is a Lipschitz map such that 
for each $n$ and each $n$-cylinder $[e_0,\dots, e_{n-1}]$, the set $\pi_\alpha^{-1}([e_0,\dots, e_{n-1}]) \cap X_\alpha$
is contained in one or two intervals of combined length $m(S) \beta^{1-n}$.
This suffices to conclude (cf.\ \cite[Lemma 5.3]{BS16}) that $X_\alpha$
and $\pi_\alpha(X_\alpha)$ have the same Hausdorff dimension
$\frac{\log d}{\log \beta}$.

Let $G_n:[0,k-\beta) \to S$, $\alpha \mapsto g_\alpha^n(0)$. Then the $U$ from above is a maximal interval
on which $G_n$ is continuous, and such that all $\alpha \in U$ have the same coding for the iterates
$G_m(\alpha)$, $1 \leq m \leq n$, with respect to the labelling of the $Z_i = Z_i(\alpha)$.
This allows us to define a coding map $\pi:A_\beta \to \Sigma$.
As is the case for $\pi_\alpha$, given any $n$-cylinder $[e_0,\dots e_{n-1}]$,
the preimage $\pi^{-1}([e_0,\dots e_{k-1}])$ consists of at most $n$ intervals, say $U^j$,
with $G_k(\cup_j U^j) = S$,
so the combined length satisfies $\frac1C \beta^{-n} \leq m(\cup U^j) \leq C \beta^{-n}$.
This gives a one-to-one correspondence between the components $J(\alpha)$ and
intervals $U$, and hence $A_\beta$ can be covered by $a_n$ such intervals.
More importantly, $\pi$ is Lipschitz, and its inverse on each $n$-cylinder has (at most $n$) 
uniformly Lipschitz branches.
It follows that $\dim_H(A_\beta) = \dim_H(\pi(A_\beta))$ (cf.\ \cite[Lemma 5.3]{BS16}), and
since $\Sigma \setminus \pi(A_\beta)$ is countable, also
$\dim_H(A_\beta) = \dim_H(\Sigma) = \frac{\log d}{\log \beta}$.
\end{proof}

\begin{remark}
Observe that if $g_\alpha^n(0^+) \in \bigcup_{i=0}^{d-1} V_i$, then $0^+$ is $n+1$-periodic 
under $g_\alpha$. Let $a_n$ be the number of periodic points under the $d$-fold doubling map
of prime period $n+1$. Therefore there are $a_n$ parameter intervals such that matching occurs at
$n+1$ iterates, and these have length $\sim \beta^{-n}$. 
If $d = 1$, then $a_n = \phi(n)$ is Euler's totient function, see Remark~\ref{rm:totient},
and for $d = 2$, $k=4$, so $\beta = 2 + \sqrt{2}$, then $(a_n)_{n \in \N}$ is
exactly the sequence  A038199 on OEIS \cite{OEIS}, see
the observation in Section~\ref{s:numerical} for $\beta = 2 + \sqrt{2}$, which in fact
holds for any other quadratic integer with $d = 2$. In fact, there is a 
fixed sequence $(a_n)_{n\in\N}$ for each value of $d \in \N$.
\end{remark}

\subsection{Case $-d$}\label{s:minusd}
Now we deal with the case $\beta = \beta(k,d) = \frac{k}{2} + \sqrt{(\frac{k}{2})^2+d}$, 
which is the leading root 
of $\beta^2-k\beta-d = 0$ for positive integers $k > d-1$.

\begin{theorem}
The bifurcation set $A_\beta$ has Hausdorff dimension $\dim_H(A_\beta) = \frac{\log d}{\log \beta}$.
\end{theorem}

\begin{proof}
Since $k < \beta < k+1$, $T_\alpha$ has either $k+1$ or $k+2$ branches depending on 
whether $\alpha < k+1-\beta$ or $\alpha \geq k+1-\beta$. More precisely,
\[ T_\alpha(0^-) - T_\alpha(0^+) = \begin{cases}
                 \beta - k =  \frac{d}{\beta} & \text{ if } \alpha \in [0,k+1-\beta) \quad \text{($k+1$ branches);}\\
                 \beta - (k+1) = \frac{d}{\beta} - 1& \text{ if } \alpha \in [k+1-\beta,1) \quad \text{($k+2$ branches).}
                              \end{cases}\]
In the first case, we have matching in the next iterate due to Lemma~\ref{l:1overbeta}.

\vskip .2cm 
Let $\gamma := (k+1)-\beta = 1-\frac{d}{\beta}$ and note that 
$1-\gamma = \frac{d}{\beta} \in \Delta(d)$, 
and $\beta\gamma  = \beta - d$. 
It follows that if $x \in \Delta(i) \cap [0,1-\gamma)$, then 
\begin{eqnarray*}
T_\alpha(x+\gamma) &=& T_\alpha(x) + \beta -d \bmod 1 \\
&=&
\begin{cases}
  T_\alpha(x) + \beta-k =  T_\alpha(x) + \frac{d}{\beta}   & \text{ if } x+\gamma \in \Delta(i+k-d);\\   
  T_\alpha(x) + \beta  - (k+1) =  T_\alpha(x) - \gamma
  & \text{ if } x+\gamma \in \Delta(i+k+1-d).  
\end{cases}
\end{eqnarray*}
In the first case we have matching in the next step, and 
the second case, the difference $\gamma$ switches to $-\gamma$. 
Similarly, if $x \in \Delta(i) \cap [\gamma,1)$, then 
\[ 
T_\alpha(x-\gamma) = T_\alpha(x) - \beta + d \bmod 1 =
\begin{cases}
  T_\alpha(x) + \gamma & \text{ if } x-\gamma \in \Delta(i-(k+1)+d);\\   
  T_\alpha(x) -  \frac{d}{\beta}  & \text{ if } x-\gamma \in \Delta(i-k+d),  
\end{cases}
\]
and in the second case, we have matching in the next iterate. The transition graph for the differences $T^j(0^-) - T^j(0^+)$ is given in Figure~\ref{fig:minusd}.

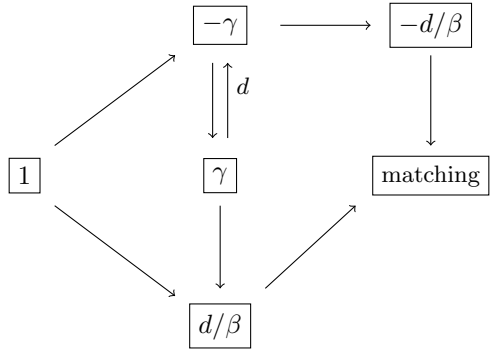
\begin{figure}[ht]
\begin{center}
\begin{tikzpicture}[scale=2]
\node[rectangle, draw=black] at (1.3,0) {$\gamma$};
\node[rectangle, draw=black] at (1.3,-1) {$d/\beta$};
\node[rectangle, draw=black] at (2.7,0) {\small matching};
\draw[->] (0.2,-.2)--(1,-.8);
\draw[->] (1.6,-.75)--(2.2,-.2);
\draw[->] (1.25,0.75)--(1.25,0.25);
\draw[->] (1.35,0.25)--(1.35,0.75);
\draw[->] (2.7,.8)--(2.7,.2);
\node[rectangle, draw=black] at (0,0) {1};
\node[rectangle, draw=black] at (1.3,1) {$-\gamma$};
\node[rectangle, draw=black] at (2.7,1) {$-d/\beta$};
\draw[->] (0.2,.2)--(1,.8);
\draw[->] (1.7,1)--(2.3,1);
\draw[->] (1.3,-0.2)--(1.3,-0.75);
\node at (1.45,.6) {\small $d$};
\end{tikzpicture}
\caption{The transition graph for the root of $\beta^2 - k\beta - d=0$. 
The vertical arrows stand for the $d$ possible $i$ that 
$T_\alpha^j(0^-) \in \Delta(i)$ and $T_\alpha^j(0^+) = T_\alpha^j(0^-) \pm \gamma \in [0,1]$.}
\label{fig:minusd}
\end{center}
\end{figure} 

As long as there is no matching, $T_\alpha$ switches the order of $T_\alpha^j(0^-)$
and $T_\alpha^j(0^+)$, and therefore we consider the second iterate with ``forbidden regions''
as the composition of two degree $d$-maps with plateaus.

For the first iterate, define for $i = 0,\dots, d-1$ the ``forbidden regions''
\[ 
V_i := \{ x \in \Delta(i) : x+\gamma \in \Delta(i+k-d) \} 
= \Big[ \frac{i+1-\alpha}{\beta} \ , \ \frac{i+1+k-d-\alpha}{\beta} - \gamma\ \Big).
\]
Note that $T_\alpha(1 - \gamma) = \alpha = T_\alpha(0^+)$ and
$T_\alpha(\frac{i+1+k-d-\alpha}{\beta} - \gamma) = \gamma$. 
Define $f^1_\alpha:[0,1-\gamma] \to [\gamma,1]$ as
\[ f^1_\alpha(x) = \begin{cases}
               \gamma & \text { if } x \in V := \bigcup_{i=0}^{d-1}  V_i; \\
               T_\alpha(x) & \text{ otherwise.}
              \end{cases}
\]
For the second iterate, the ``forbidden regions'' are
\[ 
W_i := \{ x \in \Delta(i) : x - \gamma \in \Delta(i-(k-d)) \} 
= \Big[ \frac{i-(k-d)-\alpha}{\beta} + \gamma \ , \ \frac{i+1-\alpha}{\beta}  \Big)
\]
for $i = k-d+1 ,\dots, k$.
Note that $T_\alpha(\gamma) = \beta+\alpha - (k+1) = T_\alpha(0^-)$ and 
$T_\alpha(\frac{i-(k-d)-\alpha}{\beta} + \gamma) = 1-\gamma$. 
Define $f^2_\alpha:[\gamma,1] \to [0,1-\gamma]$ as
\[ f^2_\alpha(x) = \begin{cases}
               1-\gamma & \text { if } x \in W := \bigcup_{i=k-d+1}^{k} W_i; \\
               T_\alpha(x) & \text{ otherwise.}
              \end{cases}
\]  
The composition
$g_\alpha := f^2_\alpha \circ f^1_\alpha : [0,1-\gamma] \to [0, 1-\gamma]$,
once we identify $0 \sim 1-\gamma$, becomes a non-decreasing degree $d^2$ circle endomorphism
with $d+d^2$ plateaus, and slope $\beta^2$ elsewhere.

\begin{figure}[h]
\centering
\subfigure[$T_\alpha$]{
\begin{tikzpicture}[scale=3.2]
\filldraw[yellow!30] (.171,0) rectangle (.226,1);
\filldraw[yellow!30] (.435,0) rectangle (.49,1);
\filldraw[yellow!30] (.699,0) rectangle (.754,1);
\draw(0,0)node[left]{\small $0$}--(.2,0)node[below]{\small $V_0$}--(.46,0)node[below]{\small $V_1$}--(.73,0)node[below]{\small $V_2$}--(1,0)node[below]{\small$1$}--(1,1)--(0,1)node[left]{\small $1$}--(0,.35)node[left]{\small $\alpha$}--(0,.209)node[left]{\small $\gamma$}--(0,0);
\draw[thick, purple!50!black](0,.35)--(.171,1)(.171,0)--(.435,1)(.435,0)--(.699,1)(.699,0)--(.963,1) (.963,0)--(1,.141);
\draw[dotted](0,0)--(1,1)(.171,0)--(.171,1)(.435,0)--(.435,1)(.699,0)--(.699,1)(.963,0)--(.963,1);
\draw[dashed](0,.209)--(.791,.209)--(.791,1);
\draw[dotted](.226,0)--(.226,1)(.49,0)--(.49,1)(.754,0)--(.754,1);
\end{tikzpicture}}
\hspace{2cm}
\subfigure[$T_\alpha$]{
\begin{tikzpicture}[scale=3.2]
\filldraw[yellow!30] (.908,0) rectangle (.963,1);
\filldraw[yellow!30] (.435,1) rectangle (.38,0);
\filldraw[yellow!30] (.699,0) rectangle (.644,1);
\draw(0,0)node[left]{\small $0$}--(.209,0)node[below]{\small $\gamma$}--(.405,0)node[below]{\small $W_1$}--(.67,0)node[below]{\small $W_2$}--(.93,0)node[below]{\small $W_3$}--(1,0)--(1,1)--(0,1)node[left]{\small $1$}--(0,.791)node[left]{\small $1-\gamma$}--(0,.35)node[left]{\small $\alpha$}--(0,0);
\draw[thick, purple!50!black](0,.35)--(.171,1)(.171,0)--(.435,1)(.435,0)--(.699,1)(.699,0)--(.963,1) (.963,0)--(1,.141);
\draw[dotted](0,0)--(1,1)(.171,0)--(.171,1)(.435,0)--(.435,1)(.699,0)--(.699,1)(.963,0)--(.963,1);
\draw[dashed](.209,0)--(.209,.791)--(1,.791);
\draw[dotted](.38,0)--(.38,1)(.644,0)--(.644,1)(.908,0)--(.908,1);
\end{tikzpicture}}
\hspace{1cm}

\subfigure[$f^1_\alpha$]{
\begin{tikzpicture}[scale=3.2]
\draw(0,.209)node[below]{\small $0$}--(.2,.209)node[below]{\small $V_0$}--(.46,.209)node[below]{\small $V_1$}--(.73,.209)node[below]{\small $V_2$}--(.791,.209)--(.791,1)--(0,1)node[left]{\small 1}--(0,.35)node[left]{\small $\alpha$}--(0,.209)node[left]{\small $\gamma$};
\draw[very thick, purple!50!black](0,.35)--(.171,1)(.171,.209)--(.226,.209)--(.435,1)(.435,.209)--(.49,.209)--(.699,1)(.699,.209)--(.754,.209)--(.791,.35);
\draw[dotted](0,.35)--(.791,.35)(.171,.209)--(.171,1)(.435,.209)--(.435,1)(.699,.209)--(.699,1)(.226,.209)--(.226,1)(.49,.209)--(.49,1)(.754,.209)--(.754,1);
\node at (.92,.125) {\small $1-\gamma$};
\draw[white] (1.2,.209)--(1.2,1);
\end{tikzpicture}}
\hspace{.9cm}
\subfigure[$f^2_\alpha$]{
\begin{tikzpicture}[scale=3.2]
\draw(.209,0)node[below]{\small $\gamma$}--(.405,0)node[below]{\small $W_1$}--(.67,0)node[below]{\small $W_2$}--(.93,0)node[below]{\small $W_3$}--(1,0)--(1,.791)--(.209,.791)node[left]{\small $1-\gamma$}--(.209,.141)node[left]{\small $\beta+ \alpha-(k+1)$}--(.209,0)node[left]{\small $0$};
\draw[very thick, purple!50!black](.209,.141)--(.38,.791)--(.435,.791)(.435,0)--(.644,.791)--(.699,.791)(.699,0)--(.908,.791)--(.963,.791)(.963,0)--(1,.141);
\draw[dotted](.209,.141)--(1,.141)(.38,0)--(.38,.791)(.435,0)--(.435,.791)(.699,0)--(.699,.791)(.644,0)--(.644,.791)(.908,0)--(.908,.791)(.963,0)--(.963,.791);
\draw[white] (1.4,0)--(1.4,.791);
\end{tikzpicture}}
\caption{The maps $T_\alpha$, $f_\alpha^1$ and $f_\alpha^2$ for $\beta$ satisfying $\beta^2-3\beta-2$ and $\alpha=0.35$.}
\end{figure}
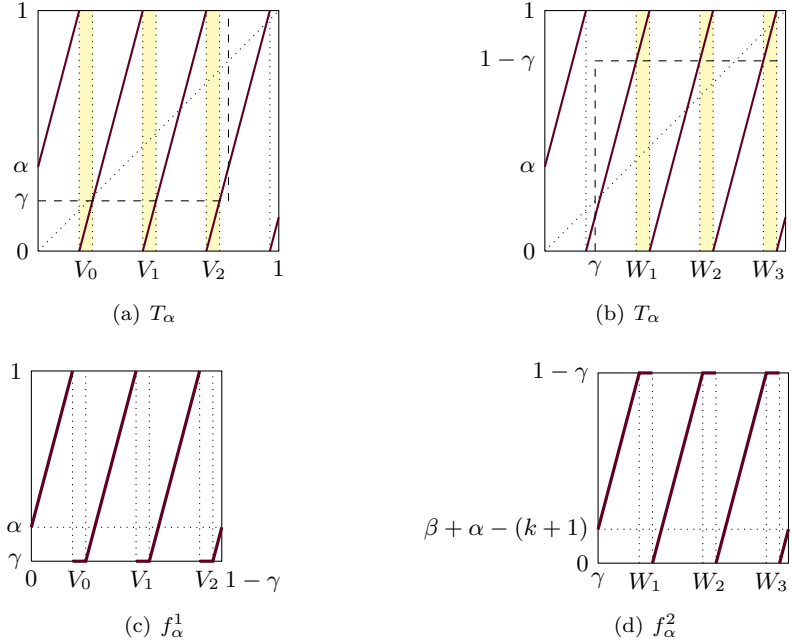

The argument in the previous case gives again that
\[ 
X_\alpha = \{ x \in [0,1-\gamma] : g_\alpha^n(x) \notin \text{ plateaus for all } n \in \N\}
\]
has Hausdorff dimension $\frac{\log d^2}{\log \beta^2} = \frac{\log d}{\log \beta}$,
and also that
$A_{\beta} = \{ \alpha \in [0, k-\beta] \, : \, g_\alpha^n(0^+) \notin \text{ plateaus} \}$ 
has $\dim_H(A_\beta) = \frac{\log d}{\log \beta}$.
\end{proof}

\section{Multinacci numbers}\label{sec:multinacci}
The last family of Pisot numbers that we consider are the multinacci numbers. Let $\beta$ be the Pisot number that satisfies $\beta^k - \beta^{k-1} - \dots -\beta-1=0$. These numbers increase and tend to $2$ as $k \to \infty$. The map $T_{\alpha}$ has either two or three branches. 

\begin{prop}\label{prop:two}
If $T_\alpha$ has two branches, i.e., $\alpha \in \big[0,\frac1{\beta^k}\big)$, then there is matching after $k$ steps.
\end{prop}
 
\begin{proof}
The map $T_\alpha$ has two branches if and only if
\[ \frac{2-\alpha}{\beta} \ge 1 \quad \Leftrightarrow \quad \alpha \le 2-\beta = 1-\frac1{\beta}-\frac1{\beta^2} - \cdots - \frac1{\beta^{k-1}}=\frac1{\beta^k}.\]
In case $\alpha \le \frac1{\beta^k}$ we have $T_{\alpha}(0^-) = \beta + \alpha -1 = \alpha + \frac1{\beta} + \frac1{\beta^2} + \cdots + \frac1{\beta^{k-1}}$. Since $T_{\alpha}(0^+) = \alpha$, 
this means that $T_{\alpha}(0^+) \in \Delta(0)$ and $T_{\alpha}(0^-) \in \Delta(1)$. Hence,
\[ T^2_{\alpha} (0^-) = \beta \alpha + 1 + \frac1{\beta} + \cdots + \frac1{\beta^{k-2}} + \alpha-1 = T^2_{\alpha} (0^+) + \frac1{\beta} + \cdots + \frac1{\beta^{k-2}}.\]
Continuing in the same way, we get that $T_{\alpha}^j(0^+) \in \Delta(0)$ and $T_{\alpha}^j(0^-) \in \Delta(1)$ for all $0 \le j \le k-2$ and that
\begin{eqnarray*} 
T^{j+1}_{\alpha} (0^-) &=& \beta(\beta^{j-1} + \cdots + \beta+1) \alpha + 1 + \frac1{\beta} + \cdots + \frac1{\beta^{k-(j+1)}} + \alpha-1 \\
&=& T^{j+1}_{\alpha} (0^+) + \frac1{\beta} + \cdots + \frac1{\beta^{k-(j+1)}}.
\end{eqnarray*}
So,
\[ T^{k-1}_{\alpha} (0^-) = \beta(\beta^{k-3} + \cdots + \beta+1) \alpha + 1 + \frac1{\beta}  + \alpha-1 = T^{k-1}_{\alpha} (0^+) + \frac1{\beta}.\]
Lemma~\ref{l:1overbeta} tells us that $T_{\alpha}$ has matching after $k$ steps.
\end{proof}

For the remainder of this section, we assume that $T_\alpha$ has three branches, so $\alpha > \frac{1}{\beta^k}$.

\begin{lemma}\label{lem:codes}
For every $j \ge 0$ we have
\[ |T^j_{\alpha}(0^+) - T^j_{\alpha}(0^-)| \in \Big\{ \frac{e_1}{\beta} + \frac{e_2}{\beta^2} + \cdots + \frac{e_k}{\beta^k} \, : \, e_1, \dots, e_k \in \{0,1\} \Big\}.\]
\end{lemma}

\begin{proof}
The proof goes by induction. Let $D_j = |T_\alpha^j(0^+)-T_\alpha^j(0^-)|$. Clearly 
$D_0 = 1 = \sum_{n=1}^k \frac{1}{\beta^n}$.
Suppose now that $D_j = \sum_{n=1}^k \frac{e_n}{\beta^n}$ for $e_n \in \{0,1\}$.
There are three cases:
\begin{enumerate}
\item $T_\alpha^j(0^+)$ and $T_\alpha^j(0^-)$ belong to the same $\Delta(i)$, so $D_j < \frac1{\beta}$. This implies that $e_1=0$ and thus $D_{j+1} =\beta D_j = \sum_{n=2}^k \frac{e_n}{\beta^{n-1}}$ has the required form.
\item $T_\alpha^j(0^+)$ and $T_\alpha^j(0^-)$ belong to adjacent $\Delta(i)$'s. There are two cases:\\
If $e_1 = 0$, then $D_{j+1} = 1-\beta D_j = \sum_{n=2}^k \frac{1-e_n}{\beta^{n-1}} + \frac{1}{\beta^k}$ has the required form.\\
If $e_1 = 1$, then $D_{j+1} = \beta D_j-1 = \sum_{n=2}^k \frac{e_n}{\beta^{n-1}}$ has the required form.
\item $T_\alpha^j(0^+) \in \Delta(0)$ and $T_\alpha^j(0^-)\in \Delta(2)$ or vice versa. Then $D_j > \frac{1}{\beta}$, and since $\sum_{n=2}^k \frac{1}{\beta^n} = 1-\frac1{\beta} < \frac1{\beta}$ we must have $e_1 = 1$. 
But then $D_{j+1} = 2-\beta D_j = 1-\sum_{n=2}^k \frac{e_n}{\beta^{n-1}}
= \sum_{n=2}^k \frac{1-e_n}{\beta^n} + \frac{1}{\beta^k}$ has the required form.
\end{enumerate}
This concludes the induction and the proof.
\end{proof}

\begin{remark}\label{rem:codes}
Note also that in case (i) and (ii) with $e_1 = 1$ in the above proof, $T_\alpha^j(0^+)-T_\alpha^j(0^-)$ and $T_\alpha^{j+1}(0^+)-T_\alpha^{j+1}(0^-)$ have the same sign, whereas in case (ii) with $e_1 = 0$ and case (iii), $T_\alpha^j(0^+)-T_\alpha^j(0^-)$ and $T_\alpha^{j+1}(0^+)-T_\alpha^{j+1}(0^-)$ have the opposite sign. This knowledge will be used in the proof of Theorem~\ref{thm:tribonacci2}. 
\end{remark}

Now assume that $\beta$ is the tribonacci number, i.e., the Pisot number with minimal polynomial $\beta^3-\beta^2-\beta-1$. Figure~\ref{fig:tribonacci} illustrates Lemma~\ref{lem:codes} and Remark~\ref{rem:codes}. Given $\alpha$ and initial difference $|0^--0^+| = 1=\frac1{\beta}+ \frac1{\beta^2}+\frac1{\beta^3}$ (coded as $111$), the path through the diagram is uniquely determined by the orbit of $0^+$. If $100$ is reached, there is matching in the next step.

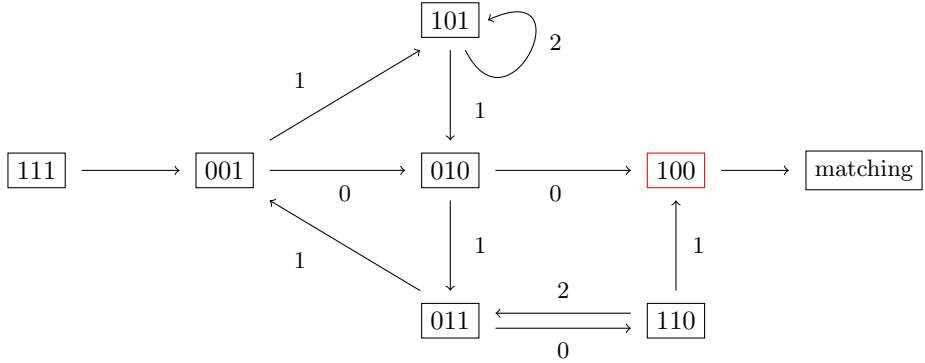
\begin{figure}[ht]
\begin{center}
\begin{tikzpicture}[scale=2]
\node[rectangle, draw=black] at (-.75,0) {111};
\node[rectangle, draw=black] at (.5,0) {001};
\node[rectangle, draw=black] at (2,0) {010};
\node[rectangle, draw=black] at (2,1) {101};
\node[rectangle, draw=black] at (2,-1) {011};
\node[rectangle, draw=black] at (3.5,-1) {110};
\node[rectangle, draw=red] at (3.5,0) {100};
\node[rectangle, draw=black] at (4.75,0) {\small matching};
\draw[->] (-.45,0)--(.2,0);
\draw[->] (.8,0)--(1.7,0);
\draw[->] (2.3,0)--(3.2,0);
\draw[->] (2,.8)--(2,.2);
\draw[->] (2,-.2)--(2,-.8);
\draw[->] (2.3,-1.05)--(3.2,-1.05);
\draw[->] (3.2,-.95)--(2.3,-.95);
\draw[->] (1.8,-.8)--(.8,-.2);
\draw[->] (.8,.2)--(1.8,.8);
\draw[->] (3.8,0)--(4.25,0);
\draw[->] (3.5,-.8)--(3.5,-.2);
\draw[->] (2.1,0.8) .. controls (2.4,.2) and (2.9,1.3) .. (2.25,1);
\node at (1,.6) {\small $1$};
\node at (1,-.6) {\small $1$};
\node at (1.3,-.15) {\small $0$};
\node at (2.7,-.15) {\small $0$};
\node at (2.7,.85) {\small $2$};
\node at (3.65,-.5) {\small $1$};
\node at (2.2,-.5) {\small $1$};
\node at (2.2,.4) {\small $1$};
\node at (2.75,-.8) {\small $2$};
\node at (2.75,-1.2) {\small $0$};
\end{tikzpicture}
\caption{The transition graph for the tribonacci number $\beta$
(root of $\beta^3 = \beta^2 + \beta + 1$).
The small numbers near the arrows indicate the difference in branch
between $T^n_\alpha(0^+)$ and $T^n_\alpha(0^-)$ when this arrow is taken.}
\label{fig:tribonacci}
\end{center}
\end{figure}

We will prove that the bifurcation set has $0<\dim_H(A_\beta) < 1$,
but first we indicate another matching interval.

\begin{prop}\label{p:notdense}
Let $\beta$ be the tribonacci number, so $\beta^3 = \beta^2 +\beta + 1$. If $\alpha \in \big[ \frac1{\beta}, \frac1{\beta^2}+\frac2{\beta^3}\big]$, then there is matching after four steps.
\end{prop}

\begin{proof}
Let $p(\alpha) = \frac{1-\alpha}{\beta-1}$ be the fixed point of the map $T_{\alpha}$. Note that  $T_{\alpha}(0^+)= \alpha > p(\alpha)$ if and only if $ \alpha > \frac1{\beta}$ and that $T_{\alpha}(0^-)= \beta + \alpha -2 < p(\alpha)$ if and only if $\alpha < \frac1{\beta^2} + \frac2{\beta^3}$. So if  $\alpha \in \big[ \frac1{\beta}, \frac1{\beta^2}+\frac2{\beta^3}\big]$, then the absolutely continuous invariant measure for $T_{\alpha}$ is not fully supported. One can check by direct computation that $T_{ \alpha}(0^+), T_{\alpha}(0^-), T^2_{\alpha}(0^+), T^2_{\alpha}(0^-) \in \Delta(1)$, which gives the following orbits:
\[ \begin{array}{ccccccc}
0^+ & \to &\alpha & \to & (\beta+1)\alpha -1 & \to & \beta^3 \alpha -\beta -1,\\
0^- & \to & \beta + \alpha -2 & \to & (\beta+1) \alpha -1-\frac1{\beta^2} & \to & \beta^3\alpha - \beta -1 -\frac1{\beta}.
\end{array}\]
Hence, there is matching after four steps. Also for $\alpha \in \big\{ \frac1{\beta}, \frac1{\beta^2}+\frac2{\beta^3}\big\}$ one can easily show that there is matching in four steps, since either $0^+$ or $0^-$ is mapped to the fixed point directly and the other one is mapped to 0 after three steps.
\end{proof}

\begin{theorem}\label{thm:tribonacci2}
If $\beta$ is the tribonacci number, 
then the bifurcation set $A_\beta$ has Hausdorff dimension strictly
between $0$ and $1$.
\end{theorem}

\begin{proof}
We know from Propositions~\ref{prop:two} and \ref{p:notdense} that there is matching if $\alpha \in [0,\frac{1}{\beta^3}]$ and $\alpha \in [\frac{1}{\beta}, \frac{1}{\beta^2} + \frac2{\beta^3}]$. Therefore is suffices to consider $\alpha \in A := [\frac1{\beta^3}, \frac1{\beta}] \cup [\frac{1}{\beta^2} + \frac2{\beta^3}, 1]$.

Figure~\ref{fig:tribonacci} enables us to represent the phase space
containing the orbits $\{ T_\alpha^n(0^+), T_\alpha^n(0^-) \}_{n \in \N}$
 as $\cup_{e \in \F} I_e$, where the {\em fiber}
$\F = \{ \pm 001, \pm 010, \pm 011, \pm 101, \pm 110 \}$ consists of ten 
points, and
$$
I_e = \begin{cases}
\, [0, 1-\frac{e_1}{\beta} - \frac{e_2}{\beta^2} - \frac{e_3}{\beta^3}] \times \{ e \} 
&\text{ if } e = +e_1e_2e_3,   \\ 
\qquad [\frac{e_1}{\beta} + \frac{e_2}{\beta^2} + \frac{e_3}{\beta^3}, 1] \times \{ e \} &\text{ if } e = -e_1e_2e_3.
\end{cases} 
$$
(Taking these subintervals of $[0,1]$ allows both
$T_\alpha^n(0^+)$ and  $T_\alpha^n(0^-) = T_\alpha^n(0^+) \pm(\frac{e_1}{\beta} + \frac{e_2}{\beta^2} + \frac{e_3}{\beta^3})$
to belong to $[0,1]$ together.)
The dynamics on this space is a skew-product
$$
S_\alpha(x, \pm e_1e_2e_3) = \Big( T_\alpha(x), \phi_\alpha(\pm e_1e_2e_3) \Big),
\qquad \phi_\alpha: \begin{array}{llccc}
e & & 0 & 1 & 2 \\
\hline
\pm001 & \mapsto & \pm010 & \mp101 & \\
\pm010 & \mapsto & \pm100 & \mp011 & \\
\pm011 & \mapsto & \pm110 & \mp001 & \\
\pm101 & \mapsto & & \pm010 & \mp101  \\
\pm110 & \mapsto & & \pm100 & \mp011  
\end{array} 
$$
where the fiber map $\phi_\alpha$ is given as in Lemma~\ref{lem:codes}
and Remark~\ref{lem:codes}. 
However, see Figure~\ref{fig:skew}, there is a set $J_\alpha$ consisting of several intervals where $S$ is not defined because these sets lead to matching in two steps. For $\alpha < 1-\frac{1}{\beta}$, we have
\begin{align*}
J_\alpha = & \Big( \Big[\frac{1-\alpha}{\beta} + \frac{1}{\beta^2}, \frac{2-\alpha}{\beta}\Big)
\cup \Big[\frac{1}{\beta^2}, \frac{1-\alpha}{\beta}\Big) \Big) \times \{ -010\} \\
& \bigcup  \Big(\Big[0,\frac{1-\alpha}{\beta} - \frac{1}{\beta^2}\Big)
\cup \Big[\frac{1-\alpha}{\beta} , \frac{2-\alpha}{\beta} - \frac{1}{\beta^2}\Big) \Big) \times \{ +010\} \\
& \bigcup \Big[\frac{1}{\beta} + \frac{1}{\beta^2}, \frac{2-\alpha}{\beta}\Big) \times \{ -110 \} 
\bigcup \Big[0,  \frac{1-\alpha}{\beta} -\frac{1}{\beta^2}\Big) \times \{ +110 \},
\end{align*}
(depicted in Figure~\ref{fig:skew}), and similar values hold for $1-\frac{1}{\beta} \leq \alpha \leq 1-\frac{1}{\beta^2}$ and 
$1-\frac{1}{\beta^2} < \alpha$. We can artificially define $S$ also on $J_\alpha$ by making it affine there as well, with slope $\beta$. We denote the resulting map by $\hat S$.

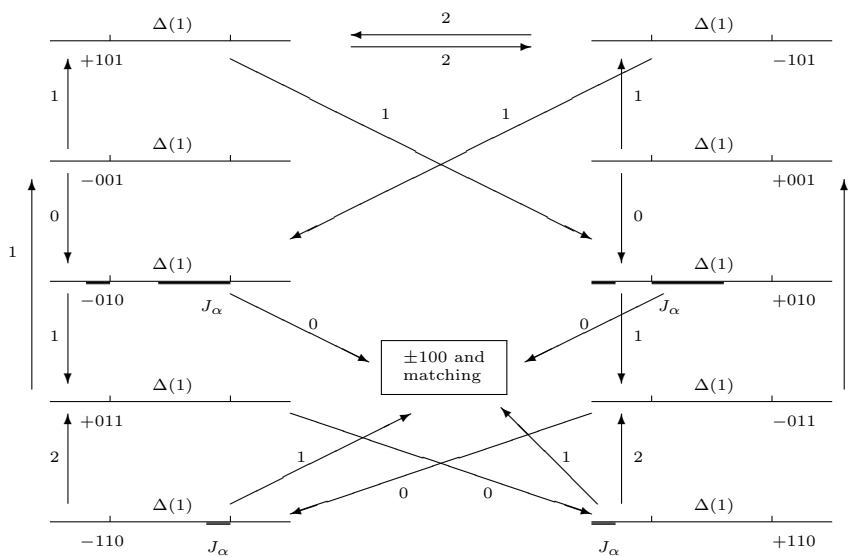
\begin{figure}[ht]
\begin{center}
\unitlength=8mm
\begin{picture}(12,10)(1.3,0) 
\put(1,1){\line(1,0){4}}\put(1.5, 0.6){\tiny$-110$} 
\put(4,0.98){\line(-1,0){0.4}}\put(4,0.96){\line(-1,0){0.4}}
\put(3.6, 0.5){\tiny$J_\alpha$}
\put(2,1){\line(0,1){0.1}}
\put(4,1){\line(0,1){0.1}}\put(2.7, 1.2){\tiny$\Delta(1)$}
\put(1.3,1.3){\vector(0,1){1.5}}\put(1, 2){\tiny$2$}
\put(5,2.8){\vector(3,-1){5}}\put(6.8, 1.4){\tiny$0$}
\put(10,1){\line(1,0){4}}\put(13, 0.6){\tiny$+110$}
\put(10,0.98){\line(1,0){0.4}}\put(10,0.96){\line(1,0){0.4}} 
\put(10.1, 0.5){\tiny$J_\alpha$}
\put(11,1){\line(0,1){0.1}}
\put(13,1){\line(0,1){0.1}}\put(11.7, 1.2){\tiny$\Delta(1)$} 
\put(10.5,1.3){\vector(0,1){1.5}}\put(10.7, 2){\tiny$2$}
\put(10,2.8){\vector(-3,-1){5}}\put(8.2, 1.4){\tiny$0$}
\put(1,3){\line(1,0){4}}\put(1.5, 2.6){\tiny$+011$} 
\put(2,3){\line(0,1){0.1}}
\put(4,3){\line(0,1){0.1}}\put(2.7, 3.2){\tiny$\Delta(1)$}
\put(0.7,3.2){\vector(0,1){3.5}}\put(0.3,5.4){\tiny$1$}
%
\put(10,3){\line(1,0){4}}\put(13, 2.6){\tiny$-011$} 
\put(11,3){\line(0,1){0.1}}
\put(13,3){\line(0,1){0.1}}\put(11.7, 3.2){\tiny$\Delta(1)$} 
\put(14.2,3.2){\vector(0,1){3.5}}\put(14.4,5.4){\tiny$1$}
\put(6.5,3.5){\tiny\fbox{$\begin{array}{l}\pm100\text{ and} \\
\text{matching} \end{array}$}}
\put(4,1.3){\vector(2,1){3}}\put(5.1,2){\tiny$1$}
\put(10.1,1.3){\vector(-1,1){1.6}}\put(9.5,2){\tiny$1$}
\put(4,4.8){\vector(2,-1){2.3}}\put(5.3, 4.2){\tiny$0$}
\put(11.2,4.8){\vector(-2,-1){2.3}}\put(9.8,4.2){\tiny$0$} 
\put(1,5){\line(1,0){4}}\put(1.5, 4.6){\tiny$-010$} 
\put(4,4.98){\line(-1,0){1.2}}\put(4,4.96){\line(-1,0){1.2}} 
\put(2,4.98){\line(-1,0){0.4}}\put(2,4.96){\line(-1,0){0.4}} 
\put(3.5, 4.5){\tiny$J_\alpha$}
\put(2,5){\line(0,1){0.1}}
\put(4,5){\line(0,1){0.1}}\put(2.7, 5.2){\tiny$\Delta(1)$}
\put(1.3,4.8){\vector(0,-1){1.5}}\put(1, 4){\tiny$1$}
\put(10,5){\line(1,0){4}}\put(13, 4.6){\tiny$+010$} 
\put(11,4.98){\line(1,0){1.2}}\put(11,4.96){\line(1,0){1.2}} 
\put(10,4.98){\line(1,0){0.4}}\put(10,4.96){\line(1,0){0.4}} 
\put(11.1, 4.5){\tiny$J_\alpha$}
\put(11,5){\line(0,1){0.1}}
\put(13,5){\line(0,1){0.1}}\put(11.7, 5.2){\tiny$\Delta(1)$} 
\put(10.5,4.8){\vector(0,-1){1.5}}\put(10.7, 4){\tiny$1$}
\put(1,7){\line(1,0){4}}\put(1.5, 6.6){\tiny$-001$} 
\put(2,7){\line(0,1){0.1}}
\put(4,7){\line(0,1){0.1}}\put(2.7, 7.2){\tiny$\Delta(1)$}
\put(1.3,7.2){\vector(0,1){1.5}}\put(1, 8){\tiny$1$}
\put(1.3,6.8){\vector(0,-1){1.5}}\put(1, 6){\tiny$0$}
\put(10,7){\line(1,0){4}}\put(13, 6.6){\tiny$+001$} 
\put(11,7){\line(0,1){0.1}}
\put(13,7){\line(0,1){0.1}}\put(11.7, 7.2){\tiny$\Delta(1)$}
\put(10.5,7.2){\vector(0,1){1.5}}\put(10.7, 8){\tiny$1$} 
\put(10.5,6.8){\vector(0,-1){1.5}}\put(10.7, 6){\tiny$0$}
\put(1,9){\line(1,0){4}}\put(1.5, 8.6){\tiny$+101$} 
\put(2,9){\line(0,1){0.1}}
\put(4,9){\line(0,1){0.1}}\put(2.7, 9.2){\tiny$\Delta(1)$}
\put(6,8.9){\vector(1,0){3}}\put(7.5, 8.6){\tiny$2$}
\put(4,8.7){\vector(2,-1){6}}\put(8.5, 7.7){\tiny$1$}
\put(10,9){\line(1,0){4}}\put(13, 8.6){\tiny$-101$} 
\put(11,9){\line(0,1){0.1}}
\put(13,9){\line(0,1){0.1}}\put(11.7, 9.2){\tiny$\Delta(1)$}
\put(9,9.1){\vector(-1,0){3}}\put(7.5, 9.3){\tiny$2$}
\put(11,8.7){\vector(-2,-1){6}}\put(6.5, 7.7){\tiny$1$}  
\end{picture}
\caption{The central part of Figure~\ref{fig:tribonacci}
represented as expanding map on ten intervals.
The small bold 
subintervals in the intervals encoded $\pm010$ and $\pm110$
are, for $\alpha < 1-\frac{1}{\beta}$, the regions leading to matching in two steps.}
\label{fig:skew}
\end{center}
\end{figure}

{\bf Claim:} For each $\alpha \in A$ the set $\cup_{n \geq 0} S_\alpha^{-n}(J_\alpha)$ is dense in $[0,1]$.

To prove the claim, first observe that the choice of $\alpha$ ensures 
that $T_\alpha$ has a  fully supported invariant density.
Therefore, by the Ergodic Theorem, the orbit of Lebesgue-a.e.\ point is dense in $[0,1]$.
In particular, for $p = \frac{1-\alpha}{\beta-1}$ the fixed point,
 $\cup_{n \geq 0} T_\alpha^{-n}(p)$ is dense.
This implies that $\cup_{n \geq 0} \hat S_\alpha^{-n}(\{ p \} \times \F)$
is dense in $\cup_{e \in\F} I_e$.
By inspection of Figure~\ref{fig:skew} we have that by only
using arrows labelled $0$ and $1$,
state $\pm 010$ cannot be avoided for more than four iterates.
Hence for every neighbourhood
$V \owns (p,e)$ for some $e \in \F$, there is an $n \in \N$ such that
$\hat S^n(V) \cap J_\alpha \neq \emptyset$.
This together proves the claim.
\\[3mm]
Let $X_\alpha := \{ x \in \cup_{e \in\F} I_e : S_\alpha^n(x) \notin J_\alpha\ \forall n \geq 0\}$.
It follows from \cite[Theorem 1 \& 2]{Rai92}
(extending results of \cite{MS80} and \cite{Urb87})
that $h_{top}(S_\alpha|X_\alpha)$ and $h_{top}(\hat S_\alpha|X_\alpha)$
depend continuously on $\alpha \in A$. 
In fact, \cite[Theorem 3]{Rai92} gives that
$\dim_H(X_\alpha)$ depends
continuously on $\alpha$ as well, and therefore
they reach their infimum and supremum for some $\alpha \in A$.

{\bf The upper bound:} To show that  $\sup_{\alpha \in A} \dim_H(X_\alpha) < 1$, note 
that, since $\hat S_\alpha$ a uniformly expanding with slope $\beta$, 
it is Lebesgue ergodic, with a unique measure of maximal 
entropy $\mu_{\max}(\alpha)$.
This measure is fully supported, and clearly $h_{top}(\hat S_\alpha) = \log \beta$.
Given that $\supp(\mu_{\max}(\alpha)) \neq X_\alpha$,
$h_{top}(\hat S_\alpha|X_\alpha) < \log \beta$ and the above mentioned
continuity implies that also 
$\sup_{\alpha \in A} h_{top}(\hat S_\alpha|X_\alpha) < \log \beta$.
The dimension formula (\cite{Bow79}) implies that 
$$
\sup_{\alpha \in A} \dim_H(X_\alpha) \leq
\sup_{\alpha \in A} \frac{h_{top}(\hat S_\alpha|X_\alpha)}
{\chi(\hat S_\alpha|X_\alpha)}
=\sup_{\alpha \in A} \frac{h_{top}(\hat S_\alpha|X_\alpha)}{\log \beta}
< 1.
$$
Here $\chi$ denotes the Lyapunov exponent, which has the common value
$\log \beta$ for every  $\hat S_\alpha$-invariant measure, 
because the slope is constant $\beta$.

If ${\mathcal V} = {\mathcal V}(\alpha) = \{ V_k \}_k$ is a cover of $X_\alpha$ consisting of closed disjoint intervals,
then we can shrink each $V_k$ so that $\partial V_k$ consists of 
preimages $w(\alpha)$ of $\Big(\!\{\frac{1-\alpha}{\beta}, \frac{2-\alpha}{\beta}\} \times \F\Big)
\cup\, \bigcup_{e \in \F} \partial I_e \cup \partial J_\alpha$.
By Remark~\ref{rem:pw_affine}, the intervals $V_k$ move to the left with speed
$\approx \frac{1}{\beta-1}$ as $\alpha$ moves in $A$.
At the same time, $S_\alpha(0^+ , \{ -001\})$ moves with speed $1$ to the right.
Therefore, to each $V_k = V_k(\alpha) \in {\mathcal V}(\alpha)$ corresponds a parameter interval $U_k$ of
length $\leq 2(1+\frac{1}{\beta-1}) |V_k|$ such that $T_\alpha(0^+) \in V_k(\alpha)$ 
if and only if $\alpha \in U_k$, and
the $\gamma$-dimensional Hausdorff mass satisfies $\sum_k |U_k|^\gamma \leq 
 2^\gamma(1+\frac{1}{\beta-1})^\gamma \sum_k |V_k|^\gamma$.
Furthermore, $A_\beta \subset \cup_k U_k$ for each cover $\mathcal V$,
and therefore $\dim_H(A_\alpha) \leq \sup_{\alpha \in A} \dim_H(X_\alpha) < 1$.

{\bf The lower bound:} To obtain a positive lower bound we find a Cantor subset $K(\alpha) \subset X_\alpha$ such that
$h_{top}(S_\alpha|K(\alpha))$ can be estimated. Let $e \in \F$ and $p_0 = p_0(\alpha) \in I_e$ be periodic under $S_\alpha$, say of period $m$. We think of $p_0$ as a lift of the fixed point $p$ of $T_\alpha$, but due to the transition between fibers, the period of $p_0$ is strictly larger than $1$. Take $\alpha_0$ such that $S^{n-1}_{\alpha_0}(\alpha_0, -001) = p_0$ for some $n \in \N$, and let $U \owns \alpha_0$ be a neighbourhood such that $g: \alpha \mapsto S_\alpha^{n-1}(\alpha,-001)$ maps $U$ to a fixed neighbourhood of $p_0$; more precisely, write $U = (\alpha_1, \alpha_2)$ and assume that there is $\eta > 0$ such that $g|U$ is monotone, $g(\alpha_1) = p_0(\alpha_1) - \eta \in I_e$ and  $g(\alpha_2) = p_0(\alpha_2) + \eta \in I_e$. Since $\frac{d}{d\alpha} g(\alpha) = \frac{\beta^n-1}{\beta-1}$ and $\frac{d}{d\alpha} p_0(\alpha) = \frac{-1}{\beta}$ by Remark~\ref{rem:pw_affine}, such $\eta > 0$ and interval $U$ can be found.

Since $\cup_{j \geq 0} S_\alpha^{-j}(p_0(\alpha))$ is dense, we can find $q_0 \in I_e \setminus \{ p_0\}$ such that $S_\alpha^k(q_0) = p_0$ for some $k \geq 1$. Moreover, the preimages $q_{-j} = I_e \cap S_\alpha^{-m}(q_{1-j})$, $j \geq 1$, converge exponentially fast to $p_0$ as $j \to \infty$. Let $V_0$ be a neighbourhood of $q_0$ such that $S_\alpha^k$ maps $V_0$ diffeomorphically onto a neighbourhood $W \owns p_0$; say $W \subset B_\eta(p_0)$. Let $V_{-j} = I_e \cap S_\alpha^{-m}(q_{1-j})$ and find $j'$ so large that $V_{-j'}$ and $V_{-j'-1} \subset W$. Then there is a maximal Cantor set $K' \subset V_{-j'} \cup V_{-j'-1}$ such that $S_\alpha^{k+j'm}(V_{-j'} \cap K') = S_\alpha^{k+(j'+1)m}(V_{-j'-1} \cap K')= K'$. Next let $K = \cup_i S_\alpha^i(K')$. Then $K = K(\alpha)$ is $S_\alpha$-invariant and since $K'$ is in fact a full horseshoe of two branches with periods $k+j'm$ and $k+(j'+1)m$, $h_{top}(S_\alpha|K) \geq \frac{\log 2}{k+(j'+1)m} > 0$. All these sets vary as $\alpha$ varies in $U$, but since the slope is constant $\beta$, the dimension formula gives 
$\dim_H(K(\alpha)) \geq \frac{\log 2}{(k+(j'+1)m) \log \beta}$, uniformly in $\alpha \in U$.

The set $K(\alpha)$ moves slowly as $\alpha$ moves in $A$; as in Remark~\ref{rem:pw_affine}, $\frac{d}{d\alpha} y(\alpha)$ is bounded when $y(\alpha) \in K(\alpha)$ is the continuation of $y \in K(\alpha_0)$. At the same time, $g$ is affine on $U$ and has slope $\frac{d}{d\alpha} T^{n-1}_\alpha(\alpha) \equiv \frac{\beta^n-1}{\beta-1}$. Therefore, $U$ contains a Cantor set of dimension $\ge \frac{\log 2}{(k+(j'+1)m)\log \beta}$, and this proves the lower bound.
\end{proof}

\begin{figure}[ht]
\begin{center}
\begin{tikzpicture}[scale=2]
\node[rectangle, draw=black] at (-1,1) {1111};
\node[rectangle, draw=black] at (0,0) {0001};
\node[rectangle, draw=black] at (1,0) {0100};
\node[rectangle, draw=red] at (2,0) {1000};
\node[rectangle, draw=black] at (3,0) {1100};
\node[rectangle, draw=black] at (2,1) {0111};
\node[rectangle, draw=black] at (3,1) {1110};
\node[rectangle, draw=black] at (-1,-1) {1101};
\node[rectangle, draw=black] at (0,-1) {0101};
\node[rectangle, draw=black] at (1,-1) {0010};
\node[rectangle, draw=black] at (3,-1) {0110};
\node[rectangle, draw=black] at (0,-2) {1010};
\node[rectangle, draw=black] at (1,-2) {1011};
\node[rectangle, draw=black] at (2,-2) {1001};
\node[rectangle, draw=black] at (3,-2) {0011};
\node[rectangle, draw=black] at (2,-.75) {\small matching};
\draw[->] (-1,.8)--(-.3,.2);
\draw[->] (.3,-.2)--(.7,-.8);
\node at (.45,-.2) {0};
\draw[->] (-.3,-.2)--(-1,-.8);
\node at (-.65,-.3) {1};
\draw[->] (2,-.2)--(2,-.55);

\draw[->] (-.7,-1)--(-.3,-1);
\node at (-.5,-.85) {2};
\draw[->] (-.7,-1.2)--(-.3,-1.8);
\node at (-.65,-1.7) {1};

\draw[->](0,-1.2)--(0,-1.8);
\node at (.15,-1.5) {0};
\draw[->](.3,-2)--(.7,-2);
\node at (.5,-2.15) {2};
\draw[->](1.3,-2)--(1.7,-2);
\node at (1.5,-2.15) {2};
\draw[->] (2.7,-2)--(2.3,-2);
\node at (2.5,-2.15) {1};
\draw[->](1,-.8)--(1,-.2);
\node at (1.15,-.5) {0};

\draw[->](1,-1.2)--(1,-1.8);
\node at (1.15,-1.5) {1};
\draw[->](2.9,-1.2)--(2.9,-1.8);
\node at (2.75,-1.5) {1};
\draw[->](3.1,-1.8)--(3.1,-1.2);
\node at (3.25,-1.5) {0};
\draw[->](3,-.8)--(3,-.2);
\node at (3.15,-.5) {0};
\draw[->](3,.8)--(3,.2);
\node at (3.15,.5) {1};

\draw[->](1.3,0)--(1.7,0);
\node at (1.5,-.15) {0};
\draw[->](2.7,0)--(2.3,0);
\node at (2.5,-.15) {1};
\draw[->](2.3,1)--(2.7,1);
\node at (2.5,1.15) {0};

\draw[->] (2.7,.2)--(2.3,.8);
\node at (2.6,.6) {2};
\draw[->] (1,.2)--(1.7,.8);
\node at (1.25,.6) {1};
\draw[->](2,-1.8)--(1.3,-1.2);
\node at (1.75,-1.4) {1};

\draw[->](1.7,1)--(0,1)--(0,.2);
\node at (1,1.15) {1};
\draw[->](3.3,1)--(3.5,1)--(3.5,-2)--(3.3,-2);
\node at (3.65,-.5) {2};
\draw[->](2,-2.2)--(2,-2.4)--(-1,-2.4)--(-1,-1.2);
\node at (-1.15,-1.7) {2};

\draw[->] (.1,-.8) .. controls (.4,-.2) and (-.4,-.2) .. (-.1,-.8);
\node at (0,-.5) {1};
\draw[->] (.3,-1.8)--(.7,-.2);
\node at (.5,-1.4) {1};
\draw[->](1.3,-1.8)--(2.7,-1.2);
\node at (2.3,-1.55) {1};
\end{tikzpicture}
\caption{The transition graph for the Pisot number $\beta$ that is the root of $\beta^4 = \beta^3 + \beta^2 + \beta + 1$.
The small numbers near the arrows indicate the difference in branch between $T^n_\alpha(0)$ and $T^n_\alpha(1)$ when this arrow needs to be taken.}
\label{fig:4bonacci}
\end{center}
\end{figure}
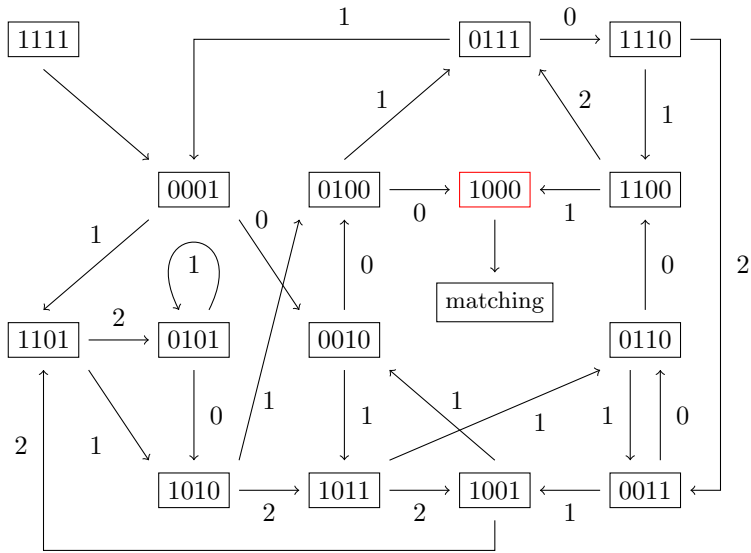

As the degree of the multinacci number becomes higher, the transition graph becomes more complicated. Figure~\ref{fig:4bonacci} shows the graph corresponding to the Pisot number with minimal polynomial $\beta^4 - \beta^3-\beta^2-\beta-1$. We conjecture that there is a.s.~matching for all multinacci numbers $\beta$, but it is not obvious how the technique from Theorem~\ref{thm:tribonacci2} can be extended to cover the general case.

\section{Final remarks}
We make one final remark about the presence of matching for Pisot and non-Pisot $\beta$'s. The techniques used in this article are based on the fact that the difference between $T^n_\alpha(0^+)$ and $T^n_\alpha(0^-)$ can only take finitely many values. Note that $|T^n_\alpha(0^+)-T^n_\alpha(0^-)| = \sum_{i=0}^n c_i \beta^i$ for some $c_i \in \{0,\pm 1, \ldots, \pm \lceil \beta \rceil \}$, see (\ref{q:0orbit}). The Garsia Separation Lemma (\cite{Gar62}) implies that for any Pisot number $\beta$, the set
\[ Y(\beta) = \big\{ \sum_{i=0}^n c_i \beta^i \, : \, c_i \in \{ 0, \pm 1, \ldots, \pm \lceil \beta \rceil\}, \, n = 0,1,2,\ldots \big\}\]
is uniformly discrete, i.e., there is an $r>0$ such that $|x-y|>r$ for each $x\neq y \in Y(\beta)$ (see for example \cite{AK13}). In particular, $Y(\beta)\cap [0,1]$ is finite. Hence, for any Pisot number $\beta$ the set of states in the corresponding transition graph will be finite. We suspect that a substantial number of actual paths through this graph will lead to the matching state, which leads us to the following conjecture.

\begin{conj}
If $\beta$ is a non-quadratic Pisot number, then $0 < \dim_H(A_\beta) <1$.
\end{conj}

Salem numbers require further techniques to decide on prevalent matching. For non-Salem, non-Pisot algebraic numbers finding any matching seems unlikely.

\medskip
{\bf Acknowledgement:}
The authors would like to thank Christiane Frougny and the referee, whose remarks spurred us on to sharper results, which led to Theorem~\ref{t:quadratic} in its present form. The second author is partially supported by the GNAMPA group of the ``Istituto Nazionale di Alta Matematica'' (INdAM). The third author is supported by the NWO Veni-grant 639.031.140.

\end{document}